\documentclass[a4paper,reqno,11pt]{amsart}
\usepackage[T1]{fontenc}
\usepackage[utf8]{inputenc}
\usepackage{lmodern}
\usepackage[english]{babel}
\usepackage{amsmath}
\usepackage{color}
\usepackage{amsfonts}
\usepackage{amssymb}
\usepackage{amsthm}
\usepackage{stmaryrd}

\usepackage[pdfborder={0 0 0},bookmarks=false]{hyperref}
\usepackage{pstricks}
\usepackage{epsfig}
\usepackage{auto-pst-pdf}

\usepackage{etoolbox}
\usepackage{booktabs} \newcommand{\ra}[1]{\renewcommand{\arraystretch}{#1}}

\newtheorem{thm}{Theorem}[section] 
\newtheorem{lm}[thm]{Lemma}
\newtheorem{prop}[thm]{Proposition}
\newtheorem{propdef}[thm]{Proposition-Definition}

\newtheorem{cor}[thm]{Corollary}
\newtheorem{defi}[thm]{Definition}
\theoremstyle{remark}
\newtheorem{rmq}[thm]{Remark}
 
\newcommand{\tr}{\operatorname{tr}}
\newcommand{\nc}{\operatorname{nc}}
\newcommand{\Tr}{\operatorname{Tr}}
\newcommand{\Hom}{\operatorname{Hom}}
\newcommand{\Diag}{\operatorname{Diag}}
\newcommand{\diff}{\mathrm{d}}
\newcommand{\Id}{\mathrm{Id}}

\newcommand{\C}{\mathbb{C}}
\newcommand{\R}{\mathbb{R}}

\newcommand{\Z}{\mathbb{Z}}

\newcommand{\Alg}{\mathbf{Alg}}
\newcommand{\Gr}{\mathbf{Gr}}
\newcommand{\op}{\mathrm{op}}

\title{Haar states and Lévy processes on the unitary dual group}
\author[G. Cébron \and M. Ulrich]{Guillaume Cébron \and Michaël Ulrich}

\address[Guillaume Cébron]{Universität des Saarlandes (Germany)}
\email{cebron@math.uni-sb.de}
\address[Michaël Ulrich]{Université de Franche-Comté (France)\\Ernst-Moritz-Arndt Universität Greifswald (Germany)}
\email{michael.ulrich@univ-fcomte.fr}
\thanks{Guillaume Cébron is supported by the ERC advanced grant "Noncommutative distributions in free probability".}
\begin{document}

\begin{abstract}We study states on the universal noncommutative $*$-algebra generated by the coefficients of a unitary matrix, or equivalently states on the unitary dual group. Its structure of dual group in the sense of Voiculescu allows to define five natural convolutions. We prove that there exists no Haar state for those convolutions. However, we prove that there exists a weaker form of absorbing state, that we call Haar trace, for the free and the tensor convolutions. We show that the free Haar trace is the limit in distribution of the blocks of a Haar unitary matrix when the dimension tends to infinity. Finally, we study a particular class of free Lévy processes on the unitary dual group which are also the limit of the blocks of random matrices on the classical unitary group when the dimension tends to infinity.
\end{abstract}

\maketitle

\section*{Introduction}
Let $n\geq 1$, and $U(n)$ be the group of unitary $n\times n$-matrices. As proved in~\cite{Glockner1989}, the coefficient $*$-algebra of $U(n)$ generated by the matrix coordinate functions $U\mapsto U_{ij}$ is isomorphic to the commutative $*$-algebra $U_{n}$ generated by $n^2$ elements $\{u_{ij}\}_{1\leq i,j\leq n}$ fulfilling the relations which make the matrix $(u_{ij})_{ i,j=1}^n$ unitary. The group law of $U(n)$ gives rise to a structure of Hopf algebra $(U_{n},\Delta,\delta, \Sigma)$. Brown introduced in~\cite{Brown1981} the (noncommutative) $*$-algebra $U_{n}^{\nc}$, sometimes called the Brown algebra, generated by $n^2$ elements $\{u_{ij}\}_{1\leq i,j\leq n}$ fulfilling the same relations making $(u_{ij})_{ i,j=1}^n$ unitary. Even though the Brown algebra is not a Hopf algebra, which seems to limit its study, it is possible to define a structure of a dual group $U\langle n \rangle=(U_{n}^{\nc},\Delta,\delta, \Sigma)$ in the sense of Voiculescu~\cite{Voiculescu1987} (see Definition~\ref{dualVoi}). We refer to $U\langle n \rangle$ as the \emph{unitary dual group}, and $U_{n}^{\nc}$ has to been considered as the "matrix coordinate functions" on $U\langle n \rangle$. This structure allows to define naturally five notions of convolution of states on $U_{n}^{\nc}$ (of "measures" on $U\langle n \rangle$) with respect to the five natural independence: the \emph{free convolution}, the \emph{tensor convolution}, the \emph{boolean}, the \emph{monotone} and the \emph{anti-monotone} one. The original motivation for this paper was to understand the existence of Haar states, or absorbing states, for those different convolutions. In the present paper, we prove that, except in the case $n=1$, there exist no Haar states for those five convolutions.

There exist at least two different ways of representing $U\langle n \rangle$. On one hand, Glockner and von Waldenfels proved in~\cite{Glockner1989} that $U_{n}^{\nc}$ is isomorphic to the complex algebra generated by some concrete operator-valued functions.
On the other hand, in~\cite{McClanahan1992}, Mc Clanahan studies the $C^*$-algebra $C^*(U_{n}^{\nc})$ generated by $U_{n}^{\nc}$ and proves that it is isomorphic to the relative commutant of the matrix algebra $\mathcal{M}_n(\mathbb{C})$ in the free product of $\mathcal{M}_n(\mathbb{C})$ with the algebra of continuous function $C(\mathbb{U})$ on the unit complex circle. In fact, the work of Mc Clanahan covers also the study of a particular state, which is the free product state of the normalized trace on $\mathcal{M}_n(\mathbb{C})$ and the Haar measure on $C(\mathbb{U})$ in the previous construction.

We give another construction of this particular state and prove that it is a \emph{Haar trace} for the free convolution on $U\langle n\rangle$, in the sense that it is an absorbing element in the set of tracial states for the free convolution. The proof only relies on the combinatorial aspects of the free cumulant theory~\cite{Nica2006}. We construct also a Haar trace for the tensor convolution on $U\langle n\rangle$, and prove that there exist no Haar traces for the three others convolutions.

One other direction in the understanding of the unitary dual group is the study of \emph{free Lévy processes} on it, in the sense of~\cite{Ghorbal1999}. Quantum Lévy processes on quantum groups, bialgebras or dual groups have been intensively studied by Ben Ghorbal, Franz, Schürmann and Voss (see~\cite{Ghorbal2002,Ghorbal2005,Franz2006,Schurmann1993,Schurmann2014,Voss2013,Voss2013b}). Very recently, the second author of the present article enlightened a deep link between free Lévy processes on $U\langle n \rangle$ and random matrices in~\cite{Ulrich2014}. More precisely, he proved that the $n^2$ blocks of a Brownian motion on the classical unitary group $U(nN)$ of dimension $nN$ converge to the elements of a free Lévy processes on $U\langle n \rangle$ when $N$ tends to infinity. A question occurs naturally: which free Lévy processes can be obtained in the same fashion? A possible starting point is the general model of Lévy processes on the unitary group $U(nN)$ defined by the first author in~\cite{Cebron2014b}. In the present paper, we define a particular class of free Lévy processes on $U\langle n \rangle$ and prove that those processes are indeed the limit of the blocks of the model of~\cite{Cebron2014b} when $N$ tends to infinity. The proof of this phenomenon gives a new demonstration of the result in~\cite{Ulrich2014}. Moreover, the argument is also valid to construct a random matrix model for the free Haar trace on $U\langle n \rangle$. As a by-product, we prove an embedding theorem which is already well-established for the other independences: the realization of every free L\'{e}vy process on $U\langle n\rangle$ as a stochastic process on some Fock space.

The paper is organised as follows. In Section~\ref{preliminaries}, we introduce the unitary dual group $U\langle n \rangle$ as well as the different notions of convolution. We also give the description of a general construction of quantum random variables on $U\langle n \rangle$ which is useful in the others sections. In Section~\ref{Haarstate}, we state and prove Theorem~\ref{Haartracetheorem} about existence of Haar traces for the free and the tensor convolution, and non-existence for the others convolutions. In Section~\ref{matrixmodel}, we show that the free Haar trace is the limit of a Haar unitary random matrix in the sense of Theorem~\ref{limitrandomdeux}. In Section~\ref{secfree}, we introduce the free Lévy processes on $U\langle n \rangle$ and consider a particular class of processes which are limit of Lévy processes on the unitary group in the sense of Theorem~\ref{limitrandomun} and compute their generators.

\subsection*{Acknowledgements}The authors wish to thank Uwe Franz for enlightening discussions about the unitary dual group, and for algebraic insights in the proof of Theorem~\ref{Haarstatetheorem}. We also express our gratitude to the team of Roland Speicher in Saarbrücken for useful discussions related to this work, and specially Moritz Weber, whose helpful comments led to improvements in this manuscript.

\section{Preliminaries}\label{preliminaries}
This section is devoted to introduce jointly basic facts about the unitary dual group $U\langle n \rangle$ and about noncommutative probability. In Section~\ref{buildstate}, we describe a procedure which allows to define quantum random variables on $U\langle n \rangle$ starting from unitary random variables. This procedure plays a crucial role in the others sections of this paper.

\subsection{The five notions of independence}

A \emph{noncommutative space} is a unital $*$-algebra $\mathcal{A}$ equipped with a state $\phi$, that is to say a linear functional $\phi:\mathcal{A}\rightarrow\mathbb{C}$ which is positive ($\phi(A^*A)\geq0$ for all $A$ in $\mathcal{A}$), and such that $\phi(1)=1$. Let $(\mathcal{A}_1,\phi_1)$ and $(\mathcal{A}_2,\phi_2)$ be two noncommutative spaces. A $*$-homomorphism $f:\mathcal{A}_1\to\mathcal{A}_2$ such that $\phi_1=\phi_2\circ f$ is called \emph{a $*$-homomorphism of noncommutative probability spaces}.

\begin{defi}Let $\mathcal{A}$ and $\mathcal{B}$ be unital $*$-algebras. The \emph{free product} of $\mathcal{A}$ and $\mathcal{B}$ is the unique unital $*$-algebra $\mathcal{A}\sqcup \mathcal{B}$ with $*$-homomorphisms $i_1:\mathcal{A}\to \mathcal{A}\sqcup \mathcal{B}$ and $i_2:\mathcal{B}\to \mathcal{A}\sqcup \mathcal{B}$ such that, for all $*$-homomorphisms $f:\mathcal{A} \to \mathcal{C}$ and $g:\mathcal{B}\to \mathcal{C}$, there exists a unique $*$-homomorphism $f\sqcup g : \mathcal{A}\sqcup \mathcal{B} \to \mathcal{C}$ such that $f=(f\sqcup g)\circ i_1$ and $g=(f\sqcup g)\circ i_2$.
\end{defi}
%
Informally, $\mathcal{A}\sqcup \mathcal{B}$ corresponds to the "smallest" $*$-algebra containing $\mathcal{A}$ and $\mathcal{B}$ and such that there is no relation between $\mathcal{A}$ and $\mathcal{B}$ except the fact that the unit elements are identified. We usually say that $\mathcal{A}$ is the left leg of $\mathcal{A}\sqcup \mathcal{B}$, whereas $\mathcal{B}$ is its right leg, and, for all $A\in \mathcal{A}$ and $B\in \mathcal{B}$, we denote $i_1(A)$ by $A^{(1)}$ and $i_2(B)$ by $B^{(2)}$. This terminology is particularly useful when we consider the free product $\mathcal{A}\sqcup \mathcal{A}$ of $\mathcal{A}$ with itself, because, in this case, there exists two different way of thinking about $\mathcal{A}$ as a subset of $\mathcal{A}\sqcup \mathcal{A}$. Of course, if $\mathcal{A}$ and $\mathcal{B}$ are disjoint, we can avoid this subscript and identify $\mathcal{A}$ with $i_1(\mathcal{A})$ and $\mathcal{B}$ with $i_2(\mathcal{B})$.
For $*$-homomorphisms $f:\mathcal{A}\rightarrow \mathcal{C}$, $g:\mathcal{B}\rightarrow D$, we denote by $f\bigsqcup g$ the $*$-homomorphism $ (i_\mathcal{C}\circ f ) \sqcup (i_D\circ g):\mathcal{A}\sqcup \mathcal{B} \to \mathcal{C}\sqcup \mathcal{D}$.

Let $(\mathcal{A}_1,\phi_1)$ and $(\mathcal{A}_2,\phi_2)$ be two noncommutative spaces. The free product $\mathcal{A}_1\sqcup \mathcal{A}_2$ can be equipped with five different product states, called respectively free, tensor independent (or just tensor), boolean, monotone and anti-monotone product of states. We define those five constructions (see~\cite{Muraki2002} for a general study). First of all, we will assume for our unital $*$-algebras the decomposition of vector spaces $\mathcal{A}=\mathbb{C}1_{\mathcal{A}}\oplus \mathcal{A}^0$, where $\mathcal{A}^0$ is a $*$-subalgebra of $\mathcal{A}$. Remark that this decomposition is not necessarily unique, and sometimes does not exist.

\begin{defi}Let $(\mathcal{A}_1,\phi_1)$ and $(\mathcal{A}_2,\phi_2)$ be two noncommutative spaces with $\mathcal{A}_1=\mathbb{C}1_{\mathcal{A}_1}\oplus \mathcal{A}_1^0$ and $\mathcal{A}_2=\mathbb{C}1_{\mathcal{A}_1}\oplus \mathcal{A}_2^0$. There exist five different states $\phi_1\ast \phi_2$, $\phi_1\otimes \phi_2$, $\phi_1\diamond \phi_2$, $\phi_1 \rhd \phi_2$ and $\phi_1 \lhd \phi_2$ on $\mathcal{A}_1\sqcup \mathcal{A}_2$, called respectively \emph{free, tensor independent (or just tensor), boolean, monotone and anti-monotone product}, and defined, for all $A_1,\ldots, A_n\in \mathcal{A}_1\sqcup \mathcal{A}_2$ such that $A_i\in \mathcal{A}_{\epsilon_i}^0$ and $\epsilon_1\neq \epsilon_2\cdots\neq \epsilon_n$, by respectively the following relations
\begin{itemize}
\item $\phi_1\ast \phi_2(A_1\cdots A_n)=0$ whenever $\phi_1(A_1)=\cdots=\phi_n(A_n)=0$;
\item$ \phi_1\otimes \phi_2(A_1\cdots A_n)=\phi_1(\prod_{i:\epsilon_i=1}A_i)\phi_2(\prod_{i:\epsilon_i=2}A_i)$;
\item $\diamond(A_1\cdots A_n)=\prod_{i=1}^n\phi_{\epsilon_i}(A_i)$;
\item $\rhd(A_1\cdots A_n)=\phi_1(\prod_{i:\epsilon_i=1}A_i)\prod_{i:\epsilon_i=2}\phi_2(A_i)$;
\item $\lhd(A_1\cdots A_n)=\prod_{i:\epsilon_i=1}\phi_1(A_i)\phi_2(\prod_{i:\epsilon_i=2}A_i)$.
\end{itemize}
The tensor product and the free product do not depend on the choice of the decomposition $\mathcal{A}_1=\mathbb{C}1_{\mathcal{A}_1}\oplus \mathcal{A}_1^0$ and $\mathcal{A}_2=\mathbb{C}1_{\mathcal{A}_1}\oplus \mathcal{A}_2^0$, but the other three products do.
\end{defi}
The positivity follows from a GNS representation. Freeness of random variables in the sense of Voiculescu~\cite{Voiculescu1992} can be expressed thanks to this notion of free product of states.
\begin{defi}Let $A_1$ and $A_2$ be random variables in $(\mathcal{A},\phi)$. We denote by $\mathcal{A}_1$ the unital $*$-algebras generated by $A_1$, by $i_1$ the natural inclusion of $*$-algebras $i_1:\mathcal{A}_1\to \mathcal{A}$ and by $\phi_1$ the state $\phi_{|\mathcal{A}_1}$. We define similarly $\mathcal{A}_2$, $i_2$ and $\phi_2$.

We say that $A_1$ and $A_2$ are $*$-\emph{free} if
$\phi \circ (i_1\sqcup i_2)=\phi_1\ast \phi_2 $ as states on $\mathcal{A}_1\sqcup\mathcal{A}_2$.
\end{defi}
\subsection{Dual groups}
The notion of dual groups was introduced by Voiculescu~\cite{Voiculescu1987} in the 80's. We consider here the purely algebraic version. Like Hopf algebras, the idea is to generalize the notion of groups to a noncommutative setting by replacing the product by a coproduct, but we now use the free product instead of the tensor product.

\begin{defi}
A \emph{dual group} $G=(\mathcal{A},\Delta,\delta,\Sigma)$ (in the sense of Voiculescu) is a unital $*$-algebra $\mathcal{A}$, and three unital $*$-homomorphisms $\Delta:\mathcal{A}\rightarrow \mathcal{A}\sqcup \mathcal{A}$, $\delta:\mathcal{A}\rightarrow\mathbb{C}$ and $\Sigma:\mathcal{A}\rightarrow \mathcal{A}$, such that\label{dualVoi}
\begin{itemize}
\item The map $\Delta$ is a \emph{coproduct coassociative}: $(\Id\bigsqcup\Delta)\circ\Delta=(\Delta\bigsqcup \Id)\circ\Delta$
\item The map $\delta$ is a \emph{counit}: $(\delta\bigsqcup \Id)\circ\Delta=\Id=(\Id\bigsqcup\delta)\circ\Delta$
\item The map $\Sigma$ is a \emph{coinverse}: $(\Sigma\sqcup \Id)\circ\Delta=1_\mathcal{A}\circ\delta=(\Id\sqcup\Sigma)\circ\Delta$.
\end{itemize}
\end{defi}
Following the point of view of the theory of quantum groups, we consider the $*$-algebra $\mathcal{A}$ as a set of "functions on the dual group $G$", and not as the dual group. This terminology of \emph{dual group} can be ambiguous and one could prefer the terms \emph{$H$-algebras} used by Zhang in~\cite{Zhang1991}, or the term \emph{co-group} used by Bergman and Hausknecht in~\cite{Bergman1996}. However, in the following remark, whose understanding is not needed for the rest of the paper, we will see that the duality can be seen as the existence of some particular functor.

\begin{rmq}\begin{enumerate}\item Let $\Alg$ be the category of unital $*$-algebras. The dual category $\Alg^{\op}$ is the category $\Alg$ with all arrows reversed. The definition of a dual group has the immediate consequence that \emph{an element $\mathcal{A}$ in the category $\Alg$ which defines a dual group $G=(\mathcal{A},\Delta,\delta,\Sigma)$ has a group structure in the dual category $\Alg^{\op}$}, in the following sense of~\cite[Chapter 4]{Bucur1968}: we have the commutativity of all the diagrams obtained from the diagrams defining a classical group by replacing the product by $\Delta^{\op}$, the unit map by $\delta^{\op}$ and the inverse map by $\Sigma^{\op}$. Remark that $\Alg^{\op}$ is not a concrete category: the morphism $\Delta^{\op}$ in $\Alg^{\op}$ can not be seen as an actual function from $\mathcal{A}\sqcup \mathcal{A}$ to $\mathcal{A}$. Nevertheless, as shown in~\cite[Chapter 4]{Bucur1968}, this group structure is sufficient to endow naturally the set $\Hom_{\Alg^{\op}}(\mathcal{B},\mathcal{A})$ of morphisms of $\Alg^{\op}$ from any unital $*$-algebra $\mathcal{B}$ to $\mathcal{A}$ with a classical structure of group.
\item As a consequence, for any unital $*$-algebra $\mathcal{B}$, the set $\Hom_\Alg(\mathcal{A},\mathcal{B})$ of the unital $*$-homomorphisms from $\mathcal{A}$ to $\mathcal{B}$ is a group. Moreover, one can verify that $\Hom_\Alg(\mathcal{A},\cdot):\mathcal{B}\mapsto \Hom_\Alg(\mathcal{A},\mathcal{B})$ is a functor from $\Alg$ to the category of groups $\Gr$. Conversely, if a unital $*$-algebra $\mathcal{A}$ is such that $\Hom_\Alg(\mathcal{A},\cdot)$ is a functor from $\Alg$ to $\Gr$, then $G=(\mathcal{A},\Delta, \delta, \Sigma)$ is a dual group for some particular $\Delta$, $\delta$ and $\Sigma$ (see~\cite{Zhang1991} for a direct proof, or~\cite[Chapter 4]{Bucur1968} for a proof of the dual statement about $\Alg^{\op}$). We can summarize those considerations saying that \emph{dual groups are in one-to-one correspondence with the representing objects of the functors from $\Alg$ to $\Gr$}. As a comparison, commutative Hopf algebras are the representing objects of the functors from the category of unital commutative algebras to $\Gr$.
\item Now, starting from a group $G$, one can ask the following question: is there a unital $*$-algebra $\mathcal{A}$ such that $\Hom_\Alg(\mathcal{A},\cdot)$ is a functor and $\Hom_\Alg(\mathcal{A},\C)\simeq G$? If yes, there exist $\Delta$, $\delta$ and $\Sigma$ such that $(\mathcal{A},\Delta, \delta, \Sigma)$ is a dual group which can be called \emph{a dual group of $G$} (not unique). One of Voiculescu's motivation of~\cite{Voiculescu1987} was to show that a dual action of a dual group of $G$ on some operator algebra gives rise to an action of $G$ on that operator algebra. For example, the unitary dual group $U\langle n\rangle$, the principal object of our study defined subsequently, is a dual group of the classical unitary group $U(n)=\{M\in \mathcal{M}_n(\mathbb{C}):U^*U=I_n\}$ in the sense that $\Hom_\Alg(U_{n}^{\nc},\C)\simeq U(n)$.
\end{enumerate}
\end{rmq}

As explained in the introduction, the first motivation of this article is a better understanding of Haar states and Lévy processes on dual groups. We know that those objects play a crucial role in the theory of compact quantum groups, and ideas from this theory can be a guide in the study of dual groups. However, let us emphasize, in the following remark, the major differences between dual groups and compact quantum groups.

\begin{rmq}\begin{enumerate}\item Firstly, as Hopf algebras, the definition is purely algebraic: we use only the idea of $*$-algebras and we do not need to consider some $C^*$-algebra. One possible direction of research is to consider a more analytic structure on dual groups which could lead to more powerful results.
\item The second difference is that the tensor product has here been replaced by the free product. The latter is in some way "more noncommutative" because in the case of the tensor product, the two legs of the product are still commuting. If we have gained in noncommutativity, we have lost in interpretation: while a classical (compact) group could always be seen as a (compact) quantum group via the isomorphism $C(G\times G)\simeq C(G)\otimes C(G)$, we do not have such an isomorphism any more and hence classical groups cannot be seen as special cases of dual groups.
\item Finally, let us also remark that we here impose to have $*$-homomorphisms which correspond to the idea of a neutral element and inverses, whereas in the quantum case we only imposed the quantum cancellation property. We know that this cancellation property, which in the classical case yields automatically groups, is in the quantum case somewhat weaker. If we imposed in the quantum case to have a "neutral element" and "inverses" we would have only quantum groups of Kac type.
\end{enumerate}
\end{rmq}
\subsection{Unitary dual group $U\langle n \rangle$}
We introduce now the unitary dual group $U\langle n\rangle$, first considered by Brown in~\cite{Brown1981}, and which possesses naturally a structure of dual group. It has to be considered as the noncommutative analog of the classical unitary group.
\begin{defi}
Let $n\geq 1$. The \emph{unitary dual group} is the dual group $U\langle n\rangle=(U_{n}^{\nc}, \Delta,\delta,\Sigma)$ where:
\begin{itemize}
\item The universal unital $*$-algebra $U_{n}^{\nc}$ is generated by $n^2$ elements $(u_{ij})_{1\leq i,j\leq n}$ with the relations $$\sum_{k=1}^n u_{ki}^*u_{kj}=\delta_{ij}=\sum_{k=1}^n u_{ik}u_{jk}^*.$$
\item The coproduct is given on the generators by $\Delta(u_{ij})=\sum_k u_{ik}^{(1)}u_{kj}^{(2)}$.
\item The counit is given by $\delta(u_{ij})=\delta_{ij}$.
\item The antipode is given by $\Sigma(u_{ij})=u_{ji}^*$.
\end{itemize}
\end{defi}
Let us remark that the relations defining $U_{n}^{\nc}$ can be summed up by saying that $u=(u_{ij})_{1\leq i,j\leq n}$ is a unitary matrix in $\mathcal{M}_n(U_{n}^{\nc})$. We do not suppose that $\bar{u}=(u_{ij}^*)_{1\leq i,j\leq n}$ is unitary. Indeed, unlike the relations $\sum_{k=1}^n u_{ki}^*u_{kj}=\delta_{ij}=\sum_{k=1}^n u_{ik}u_{jk}^*$, the relations $\sum_{k=1}^n u_{ik}^*u_{jk}=\delta_{ij}=\sum_{k=1}^n u_{ki}u_{kj}^*$ do not pass the coproduct $\Delta$, since we cannot simplify expressions like $\Delta(\sum_k u_{ik}^*u_{jk})=\sum_{k,p,q}u_{pk}^{(2)*}u_{ip}^{(1)*}u_{jq}^{(1)}u_{qk}^{(2)}$ to $\delta_{ij}$.


A \emph{quantum random variable} on $U\langle n\rangle$ over the probability space $\mathcal{A}$ is a $*$-homomorphism $j$ from $U_{n}^{\nc}$ to $\mathcal{A}$
(this reverse terminology is the usual one when dealing with dual objects). Of course, a quantum random variable $j:U_{n}^{\nc}\to \mathcal{A}$ yields to a unitary matrix $(j(u_{ij}))_{1\leq i,j\leq n}\in \mathcal{M}_n(\mathcal{A})$, and conversely, for all matrix $(A_{ij})_{1\leq i,j\leq n}\in \mathcal{M}_n(\mathcal{A})$ which is unitary, there exists a unique $*$-homomorphism $j:U_{n}^{\nc}\to \mathcal{A}$ such that $j(u_{ij})=A_{ij}$. In a certain sense, $U\langle n\rangle$ is one possible formalism to deal with unitary elements of $\mathcal{M}_n(\mathcal{A})$.
The coproduct leads to different notions of convolution, that we sum up below. Let us remark that we can define five different convolutions of states, instead of the unique convolution of states on quantum group, given by the tensor convolution.

\begin{defi}\begin{enumerate}
\item For two quantum random variables $j_1:U_{n}^{\nc}\to \mathcal{A}$ and $j_2:U_{n}^{\nc}\to \mathcal{A}$ on $U\langle n\rangle$, we define the \emph{convolution} $j_1\star j_2:U_{n}^{\nc}\to \mathcal{A}$ to be $j_1\star j_2=(j_1\sqcup j_2) \circ \Delta$.
\item Let us consider the decomposition $U_{n}^{\nc}=\C \oplus \ker(\delta)$. For two states $\phi, \psi$ on $U_{n}^{\nc}$, we define five different states  $\phi\star_F \psi$, $\phi\star_T \psi$, $\phi\star_B \psi$, $\phi\star_M \psi$ and $\phi\star_{AM} \psi$ on $U_{n}^{\nc}$, called respectively \emph{free, tensor independent (or just tensor), boolean, monotone} and \emph{anti-monotone convolution}, and defined by respectively
$$
\phi\star_F \psi=(\phi\ast \psi)\circ \Delta,\
\phi\star_T \psi=(\phi\otimes \psi)\circ \Delta,\
\phi\star_B \psi=(\phi\diamond \psi)\circ \Delta,$$
$$\phi\star_M \psi=(\phi\rhd \psi)\circ \Delta,\
\phi\star_{AM} \psi=(\phi\lhd \psi)\circ \Delta.$$
\end{enumerate}
\end{defi}

Let us insist on the following relations in $\mathcal{M}_n(\mathcal{A})$ (where $j_1,j_2:U_{n}^{\nc}\to \mathcal{A}$ are two quantum random variables on $U\langle n \rangle$):
\begin{align*}
(\delta(u_{ij}))_{1\leq i,j\leq n}&=I_n,\\
(j_1\circ \Sigma(u_{ij}))_{1\leq i,j\leq n}&=(j_1(u_{ij}))_{1\leq i,j\leq n}^{-1},\\
(j_1\star j_2(u_{ij}))_{1\leq i,j\leq n}&=(j_1(u_{ij}))_{1\leq i,j\leq n}\cdot (j_2(u_{ij}))_{1\leq i,j\leq n}.
\end{align*}

\subsection{How to build states on $U\langle n\rangle$?}\label{buildstate}

We expose now a general method for defining quantum random variables on $U\langle n\rangle$. Consider the noncommutative probability space $\mathcal{M}_n(\mathbb{C})$ composed of matrices of dimension $n$ equipped with its normalized trace $\tr_n:=\frac{1}{n}\Tr$. Let us denote by $E_{ij}$ the usual matricial units (ie, the matrix whose entries are zero, except for the $(i,j)$-th coefficient which is $1$).

Let $A$ be a random variable in a noncommutative space $(\mathcal{A},\phi)$. One way to consider $A$ as a matrix is to count $A$ as an element of $\mathcal{M}_n(\mathcal{A})\simeq \mathcal{A}\otimes\mathcal{M}_n(\mathbb{C})$. In this way, the {$(i,j)$-th block} of $A$ is just $\delta_{ij} A$. The starting point of our reflexion is the following: there is another way to consider $A$ as a matrix. Let us denote by $ E_{11}(\mathcal{A}\sqcup \mathcal{M}_n(\mathbb{C}))E_{11}$ the $*$-subalgebra $\{E_{11}XE_{11}:X\in \mathcal{A}\sqcup \mathcal{M}_n(\mathbb{C})\}\subset \mathcal{A}\sqcup \mathcal{M}_n(\mathbb{C})$. We have the $*$-isomorphism$$\begin{array}{rcl}\mathcal{A}\sqcup \mathcal{M}_n(\mathbb{C}) & \simeq & \Big(E_{11}(\mathcal{A}\sqcup \mathcal{M}_n(\mathbb{C}))E_{11}\Big) \otimes \mathcal{M}_n(\mathbb{C}) \\X & \mapsto & \displaystyle\sum_{1\leq i,j\leq n}E_{1i}XE_{j1}\otimes E_{ij}\\
\displaystyle\sum_{1\leq i,j\leq n}E_{i1}A_{ij}E_{1j} &  \mapsfrom & \displaystyle\sum_{1\leq i,j\leq n}A_{ij}\otimes E_{ij}.
\end{array}$$
It tells us that the $(i,j)$-th blocks of $A$ viewed as an element of $\mathcal{A}\sqcup \mathcal{M}_n(\mathbb{C})$ can be defined as $E_{1i}AE_{j1}\in E_{11}(\mathcal{A}\sqcup \mathcal{M}_n(\mathbb{C}))E_{11}$. We endow the $*$-algebra $E_{11}(\mathcal{A}\sqcup \mathcal{M}_n(\mathbb{C}))E_{11}$ with the state $n(\phi\ast \tr_n)$, where we recall that $\tr_n$ is the normalized trace on $\mathcal{M}_n(\mathbb{C})$.

\begin{propdef}
For all unitary random variable $U\in \mathcal{A}$, there exists a unique quantum random variable $j_U:U_{n}^{\nc}\rightarrow E_{11}(\mathcal{A}\sqcup \mathcal{M}_n(\mathbb{C}))E_{11}$ determined by $j_U(u_{ij})=E_{1i}UE_{j1}$,\label{haarstate} which induces a state $(n\ \phi\ast \tr_n)\circ j_U$ on $U\langle n \rangle$.
\end{propdef}
\begin{proof}
It follows from the unitarity of $(E_{1i}UE_{j1})_{1\leq i,j\leq n}$. Indeed, we have 
$$\sum_{k=1}^n E_{1i}U^*E_{k1}E_{1k}UE_{j1}
= E_{1i}U^*I_nUE_{j1}=\delta_{ij}
$$
and the same for the other relation.
\end{proof}
The elements $j_U(u_{ij})=E_{1i}UE_{j1}$ have to be considered as the $(i,j)$-th blocks of $U$, and, when there is no confusion, we will denote them by $U_{ij}$. The matrix $(U_{ij})_{1\leq i,j\leq n}$ seen as an element of
$E_{11}(\mathcal{A}\sqcup \mathcal{M}_n(\mathbb{C}))E_{11} \otimes \mathcal{M}_n(\mathbb{C})\simeq \mathcal{A}\sqcup \mathcal{M}_n(\mathbb{C})$ is exactly $U\in\mathcal{A}$ seen as an element of $\mathcal{A}\sqcup \mathcal{M}_n(\mathbb{C})$, which justifies this notation.
Remark that we have $(U_{ij})^*=(U^*)_{ji}$, and that the notation $U_{ij}^*$ is ambiguous.

\subsection{Free cumulants}
The compression of random variables by a family of matrix units has been considered in different situations, and it is possible to write explicitly the free cumulants of $U_{ij}$ in terms of those of $U$. Let us first introduce briefly this notion of cumulants (we refer the reader to the book \cite{Nica2006}).

Let $S$ be a totally ordered set. A partition of the set $S$ is said to have a crossing if there exist $i,j,k,l \in S$, with $i < j < k < l$, such that $i$ and $k$ belong to some block of the partition and $j$ and $l$ belong to another block. If a partition has no crossings, it is called non-crossing. The set of all non-crossing partitions of $S$ is denoted by $NC(S)$. When $S = \left\{1, \ldots , n\right\}$, with its natural order, we will use the notation $NC (n)$. It is a lattice with respect to the fineness relation defined as follows: for all $\pi_1$ and $\pi_2\in NC(S)$, $\pi_1 \preceq \pi_2$ if every block of $\pi_1$ is contained in a block of $\pi_2$.

\begin{defi}
The collection of free cumulants $\left(\kappa_q:\mathcal{A}^q\rightarrow\mathbb{C}\right)_{q\geq 1}$ on some probability space $(\mathcal{A},\phi)$ are defined via the following relations: for all $A_1,\ldots,A_n\in \mathcal{A}$,
$$\phi(A_1\ldots A_q)=\sum_{\sigma\in NC(q)}\hspace{-0.8cm}\prod_{\hspace{1cm}\{i_1\leq \ldots\leq i_k\}\in \sigma}\hspace{-0.8cm}\kappa_k(A_{i_1},\ldots,A_{i_k})$$
where $NC(q)$ is the set of non-crossing partitions of $\left\{1,\ldots,q\right\}$.
\end{defi}

The importance of the free cumulants is in large part due to the following characterization of freeness.
\begin{prop}
Let $\left(A_i \right)_{i\in I}$ be random variables of $(\mathcal{A},\phi)$. \label{freenesscar}They are $*$-free if and only if their mixed $*$-cumulants vanish. That is to say: for all $n\geq 0$, $\epsilon_1,\ldots , \epsilon_n$ be either $\emptyset$ or $\ast$, and all $A_{i(1)}, \ldots , A_{i(n)} \in \mathcal{A}$ such that $i(1),\ldots i(n)\in I$, whenever there exists some $j$ and $j' $ with $i(j)\neq i(j')$, we have $\kappa(A_{i(1)}^{\epsilon_1}, \ldots , A_{i(n)}^{\epsilon_n}) = 0$.
\end{prop}

We are now ready to express the free cumulants of $U_{ij}=j_U(u_{ij})=E_{1i}UE_{j1}$ as defined in Proposition-Definition~\ref{haarstate} in terms of the free cumulants of $U\in \mathcal{A}$.

\begin{prop}[Theorem 14.18 of~\cite{Nica2006}]\label{cyclic}Let $U^{(1)},\ldots, U^{(m)}$ be unitary random variables of $(\mathcal{A},\phi)$. The free cumulants of $(U^{(k)})_{ij}=j_{U^{(k)}}(u_{ij})$ in the noncommutative probability space $\Big(E_{11}(\mathcal{A}\sqcup \mathcal{M}_n(\mathbb{C}))E_{11},n(\phi\ast \tr_n)\Big)$ are given as follows. Let $1\leq i_1,j_1,\ldots,i_q,j_q\leq n$ and $1\leq m_1,\ldots, m_q\leq m$.
If the indices are cyclic, i.e. if $i_l=j_{l-1}$ for $2\leq l\leq q$ and $i_1=j_q$, we have
\begin{eqnarray*}
\kappa_q\Big((U^{(m_1)})_{i_1j_1},\ldots,(U^{(m_q)})_{i_qj_q}\Big)&=&n\ \kappa_q\Big(\frac{1}{n}U^{(m_1)},\ldots,\frac{1}{n}U^{(m_q)}\Big).
\end{eqnarray*} 
If the indices are not cyclic, the left handside is equal to zero.
\end{prop}

Let us mention two basic properties about the quantum random variables $j_U:U\langle n \rangle\to E_{11}(\mathcal{A}\sqcup \mathcal{M}_n(\mathbb{C}))E_{11}$ defined in Definition~\ref{haarstate}.

\begin{prop}\label{freeness}Let $U,V\in \mathcal{A}$ be two unitary variables of $(\mathcal{A},\phi)$.
\begin{enumerate}
\item We have $j_{U^{-1}}=j_U\circ \Sigma$ and $j_{UV}=j_U\star j_V$.
\item If $U$ and $V$ are $*$-free, then, the image $*$-algebras of $j_U$ and $j_V$ are $*$-free in the noncommutative space $\left(E_{11}(\mathcal{A}\sqcup \mathcal{M}_n(\mathbb{C}))E_{11},n(\phi\ast \tr_n)\right)$.
\end{enumerate}
\end{prop}

\begin{proof}The first property follows from the relation $j_{UV}(u_{ij})=\sum_kj_U(u_{ik})j_V(u_{kj})$. The second one follows from Proposition~\ref{cyclic} and the characterization of freeness of Proposition~\ref{freenesscar}.
\end{proof}

\section{Haar state on the unitary dual group}\label{Haarstate}
In this section, we will investigate the existence of the Haar state on $U\langle n\rangle$ for the five different convolutions. Unfortunately, the definition of a Haar state on $U\langle n\rangle$ is too strong, and we need to define a weaker notion of Haar state, namely the notion of Haar trace, to have some existence results.

\begin{defi}The free (resp. tensor independent, boolean, monotone, anti-monotone) \emph{Haar state} on $U\langle n\rangle$, if it exists, is the unique state $h$ on $U_{n}^{\nc}$ such that, for all other state $\phi$ on $U_{n}^{\nc}$, we have $\phi\star_F h=h=h\star_F\phi$ (resp. the same relation for $\star_T$, $\star_B$, $\star_M$ or $\star_{AM}$).
\end{defi}
\begin{thm}\label{Haarstatetheorem}\begin{enumerate}
\item The Haar measure on $\{z\in \mathbb{C}:|z|=1\}$ is the Haar state for the free, tensor independent, boolean, monotone and anti-monotone convolution on $U\langle1\rangle$.
\item For all $n\geq 2$, there exists no Haar state on $U\langle n\rangle$ for the free, tensor independent, boolean, monotone or anti-monotone convolution.
\end{enumerate}
\end{thm}
\begin{proof}In Section~\ref{caseone}, we prove the first item. In Section~\ref{casetwo}, we prove the second item for the free and the tensor convolution. In Section~\ref{casebool}, we prove the second item  for the boolean convolution, and finally, in Section~\ref{casemono}, we prove the second item for the monotone and anti-monotone convolution.
\end{proof}
Let us define a weaker notion of Haar state. A state $\phi$ on $U_{n}^{\nc}$ is called a \emph{tracial state}, or a \emph{trace}, if, for all $a,b\in U_{n}^{\nc}$, we have $\phi(ab)=\phi(ba)$.
\begin{defi}The free (resp. tensor independent, boolean, monotone, anti-monotone) \emph{Haar trace} on $U\langle n\rangle$, if it exists, is the unique tracial state $h$ on $U_{n}^{\nc}$ such that, for all other tracial state $\phi$ on $U_{n}^{\nc}$, we have $\phi\star_F h=h=h\star_F\phi$ (resp. the same relation for $\star_T$, $\star_B$, $\star_M$ or $\star_{AM}$).
\end{defi}
Remark that a Haar state which is tracial is automatically a Haar trace.


\begin{thm}\label{Haartracetheorem}\begin{enumerate}\item For all $n\geq 2$, there exist no Haar trace on $U\langle n\rangle$ for the boolean, monotone or anti-monotone convolution.
\item For all $n\geq 1$,  there exist a Haar trace on $U\langle n\rangle$ for the free convolution, and a Haar trace on $U\langle n\rangle$ for the tensor convolution.
\end{enumerate}
\end{thm}
\begin{rmq}As nicely communicated by Moritz Weber, a careful examination of the proof of Theorem~\ref{Haartracetheorem} allows us to conclude a more general result: the free Haar trace $h$ on $U\langle n\rangle$ is such that $\phi\star_F h=h=h\star_F\phi$ for all state $\phi$ on $U_{n}^{\nc}$ such that $\phi(\sum_{k=1}^nu_{ki}u_{kj}^*)=\phi(\sum_{k=1}^nu_{ik}^*u_{jk})=\delta_{ij}$ ($1\leq i,j\leq n$), a case which includes the tracial states but not only. For example, a state which factorizes on the unitary quantum group, where $\sum_{k=1}^nu_{ki}u_{kj}^*=\sum_{k=1}^nu_{ik}^*u_{jk}=\delta_{ij}$, fulfills this condition and so is absorbed by the free Haar trace.
\end{rmq}
\begin{proof} In Section~\ref{casebool}, we prove the first item for the boolean convolution. In Section~\ref{casemono}, we prove the first item for the monotone and anti monotone convolution. In Section~\ref{freecase}, we prove the second item for the free convolution, and give a more explicit description of the free Haar trace. In Section~\ref{freecase}, we prove the second item for the tensor convolution, and give a more explicit description of the tensor Haar trace.
\end{proof}

Let us remark that one could also choose a side and ask about a \emph{right (resp. left) Haar state} for each of these independences. It would be a state $h$ such that for each state $\phi$, it holds that $h\star\phi=h$ (resp. $\phi\star h=h$). We define similarly a \emph{right (resp. left) Haar trace}. Nevertheless, the following result shows that this notion does not introduce any more generality.
\begin{prop}\label{sidestate}
Let us consider one of the five notions of independence. If $h$ is a right (resp. left) Haar state on $U\langle n\rangle$ then it is also a left (resp. right) Haar state. As well, if $h$ is a right (resp. left) Haar trace on $U\langle n\rangle$ then it is also a left (resp. right) Haar trace.\label{rlhaar}
\end{prop}
\begin{proof}
Let $h$ be a right Haar state. We define the flip $\tau$ on $U_{n}^{\nc}\sqcup U_{n}^{\nc}$ as the $*$-homomorphism such that $\tau(u_{ij}^{(1)})=u_{ij}^{(2)}$ and $\tau(u_{ij}^{(2)})=u_{ij}^{(1)}$, where the exponent $(1)$ and $(2)$ indicate if the element is in the first leg of $U_{n}^{\nc}\sqcup U_{n}^{\nc}$ or in the second leg. A simple computation on the generators $u_{ij}$ shows that $\tau\circ(\Sigma\sqcup\Sigma)\circ\Delta=\Delta\circ\Sigma$. 
Therefore, by denoting the notion of independence at hand by $\odot$, we have for all states $\phi$:
$$ h\circ \Sigma=(h\odot\phi)\circ\Delta\circ \Sigma=(h\odot\phi)\circ\tau\circ (\Sigma\sqcup \Sigma)\circ\Delta=[(\phi\circ \Sigma)\odot (h\circ \Sigma)]\circ \Delta.
$$
Because $\Sigma$ is invertible, this says exactly that $h\circ \Sigma$ is a left Haar state. But then we have:
$$h=h\star(h\circ \Sigma)=h\circ \Sigma$$
by using the right (resp. left) Haar state property of $h$ (resp. $h\circ \Sigma$). Therefore, $h=h\circ \Sigma$ is a right and left Haar state. The argument is valid when replacing $h$ and $\phi$ by tracial states since it implies that $h\circ \Sigma$ and $\phi\circ \Sigma$ are also tracial.
\end{proof}

\subsection{The Haar state in the one-dimensional case}\label{caseone}

Let us emphasize first that we identify the states on $U\langle1\rangle$ with the probability measure on $\{z\in \mathbb{C}:|z|=1\}$ via $\mu(u^k)=\int_\mathbb{U} z^kd\mu(z)$ and $\mu({u^*}^k)=\int_\mathbb{U} \bar{z}^kd\mu(z)$ for $k\in \mathbb{N}$. The Haar measure is the uniform measure on the unit circle and is given by
$h(u^k)=h({u^*}^k)=\delta_{k0}$ for $k\in \mathbb{N}$.

The free, tensor independent, boolean, monotone and anti-monotone convolutions on $U\langle1\rangle$ correspond to five different multiplicative convolutions on probability measures on $\mathbb{U}$ which have been already studied in the literature. In each of those cases, it is straightforward to prove that $h$ is absorbing.

For the free multiplicative convolution, we refer to~\cite{Voiculescu1992}, or to Section~\ref{freecase}. For the tensor independent convolution, one has just to observe that $\phi\star_T h(u^k)=\phi(u^k)h(u^k)=\delta_{k0}$.

For the Boolean, the monotone, and the anti-monotone convolutions, our references are~\cite{Bercovici2005,Franz2005,Franz2008}. Let $\mu$ be a probability measure on $\mathbb{U}$. We define the $K$-transform of $\mu$ for $|z|<1$ by
$$K_\mu(z)=\left(\int_\mathbb{U} \frac{zx}{1-zx}d\mu(x)\right)\Big/\left(\int_\mathbb{U} \frac{1}{1-zx}d\mu(x)\right).$$
Let us remark that $K_h(z)=0$. The $K$-transform of the multiplicative Boolean convolution of $\mu$ and $\nu$ is given by $\frac{1}{z}K_\mu(z)\cdot K_\nu(z)$, and consequently, $h$ is absorbing for the Boolean convolution. The $K$-transform of the multiplicative monotone (resp. anti-monotone) convolution of $\mu$ and $\nu$ is given by $K_\mu \circ K_\nu$ (resp. $K_\nu \circ K_\mu$), and consequently, $h$ is absorbing for the monotone and anti-monotone convolutions.
\subsection{The non existence of Haar state in the free and tensor cases}\label{casetwo}
In this section, we prove that there exists no free Haar state, nor tensor Haar state, for $n\geq2$.

Let us take $n\geq2$ and assume that $h$ is a free Haar state. We take $1\leq k\leq n-1$ and we consider the unitary matrix of size $2n\times 2n$ (which is a version of~\cite[Non-example 4.1]{wang1995}, attributed to Woronowicz):
$$M_k=\left(\begin{array}{ccccccccc}I_{2k-2}&|&0&0&|&0&0&|&0\\----&|&--&--&|&--&--&|&----\\0&|&0&1&|&0&0&|&0\\0&|&0&0&|&1&0&|&0\\----&|&--&--&|&--&--&|&----\\0&|&1&0&|&0&0&|&0\\0&|&0&0&|&0&1&|&0\\----&|&--&--&|&--&--&|&----\\0&|&0&0&|&0&0&|&I_{2n-2k-2}\end{array}\right).$$
For all $1\leq i,j\leq n$, we set $j_k(u_{ij})$ the $(i,j)$-th block of $M_k$ of size $2\times 2$. Because $M_k$ is unitary, $j_k$ extends to a quantum random variable $j:U_{n}^{\nc} \rightarrow \mathcal{M}_2(\mathbb{C})$. We define the state $\phi_k$ for all $a\in U_{n}^{\nc}$ as $\phi_k(a)=\langle e_2,j_k(a)e_2\rangle$, or equivalently, as the $(2,2)$-th coefficient of $j_k(a)$. Then, for every $1\leq i,j\leq n$, we have $\phi_k(u_{ik}u_{jk}^*)=0$. Let us remark that $h$ being a free Haar state, we also have
$$h(u_{ik}u_{ik}^*)=\sum_{p,q=1}^n(h\star_F\phi_k)(u_{ip}^{(1)}u_{pk}^{(2)}u_{qk}^{(2)*}u_{iq}^{(1)*})=\sum_{p,q=1}^nh(u_{ip}u_{iq}^*)\phi_k(u_{pk}u_{qk}^*)=0$$
This reasoning can be done for any $1\leq k\leq n-1$. For $k=n$ we take the matrix $M_{k}$ in the which we have exchanged the last two columns of blocks. We therefore also have $\phi_k(u_{in}u_{jn}^*)=0$ and thus $h(u_{in}u_{in}^*)=0$.
Therefore we should have:
$$\sum_{k=1}^n h(u_{ik}u_{ik}^*)=\sum_{k=1}^n0=0$$ 
which contradicts the unitarity relation $\sum_{k=1}^n u_{ik}u_{ik}^*=1$.

The same proof can be done for the tensor case as well. Indeed, the tensor independance also verifies that, for any $1\leq i,j,p,q,k\leq n$,
$$(h\star_T\phi)(u_{ip}^{(1)}u_{pk}^{(2)}u_{qk}^{(2)*}u_{iq}^{(1)*})=h(u_{ip}u_{iq}^*)\phi(u_{pk}u_{qk}^*).$$


\subsection{The boolean case}\label{casebool}
In this section, we prove that for $n\geq2$, there exist no boolean Haar state and no boolean Haar trace on $U\langle n\rangle$.


First of all, we remark the following general result: if $\phi$ and $\psi$ are two states on $U_{n}^{\nc}$ and if $a,c$ come from the left leg and $b$ from the right leg of $U_{n}^{\nc}\sqcup U_{n}^{\nc}$, then we have
$$(\phi\star_B\psi)(abc)=(\phi\star_B\psi)\Big(a(b-\delta(b))c\Big)+\delta(b)\phi(ac)=\phi(a)\psi(b-\delta(b))\phi(c)+\delta(b)\phi(ac).$$
For all state $\phi$, let us introduce the following matrices:
\begin{eqnarray*}
  N_\phi&=&(\phi(u_{ij}))_{ij}\in\mathcal{M}_n(\mathbb{C}),\\
  \bar{N}_\phi&=&(\phi(u_{ij}^*))_{ij}\in\mathcal{M}_n(\mathbb{C}),\\
  M_\phi&=&(\phi(u_{ij}^*u_{kl}))_{(i,k)(j,l)}\in\mathcal{M}_{n^2}(\mathbb{C}).
\end{eqnarray*}
Suppose that there exists a boolean Haar state $h$. Then, for any state $\phi$,
\begin{eqnarray*}
  h(u_{ij}^*u_{kl})&=&\sum_{\alpha,\beta=1}^n(h\star_B\phi)(u_{\alpha j}^{(2)*}u_{i\alpha}^{(1)*}u_{k\beta}^{(1)}u_{\beta l}^{(2)})\\
  &=&\sum_{\alpha,\beta=1}^n[\phi(u_{\alpha j}^*)h(u_{i\alpha}^*u_{k\beta}-\delta(u_{i\alpha}^*u_{k\beta}))\phi(u_{\beta l})+\delta(u_{i\alpha}^*u_{k\beta})\phi(u_{\alpha j}^*u_{\beta l})]\\
  &=&\sum_{\alpha,\beta=1}^n\phi(u_{\alpha j}^*)h(u_{i\alpha}^*u_{k\beta})\phi(u_{\beta l})-\phi(u_{ij}^*)\phi(u_{kl})+\phi(u_{ij}^*u_{kl}),
\end{eqnarray*}
which can be written
\begin{eqnarray}
M_h&=&M_h(\bar{N}_\phi\otimes N_\phi)-(\bar{N}_\phi\otimes N_\phi)+M_\phi\nonumber\\
&=&(M_h-I_{n^2})(\bar{N}_\phi\otimes N_\phi)+M_\phi\label{eqnfondamentale}
\end{eqnarray}
where $\otimes$ denotes here the tensor product (or Kronecker product) of matrices.

A measure $\mu$ on the unitary group $U(n)=\{M\in \mathcal{M}_n(\mathbb{C}):U^*U=I_N\}$ can be seen as a unique state on $U_{n}^{\nc}$ via the integration map
$$\mu(u_{i_1j_1}^{\epsilon_1}\ldots u_{i_qj_q}^{\epsilon_q})=\int_{U(n)}U_{i_1j_1}^{\epsilon_1}\ldots U_{i_qj_q}^{\epsilon_q}\diff \mu(U).$$
Let us set $\phi_1=(1/2)(\delta_{I_n}+\delta_{-I_n})$. For all $1\leq i,j,k,l\leq n$, we have $\phi_1(u_{ij})=\phi_1(u_{ij}^*)=0$ and $\phi_1(u_{ij}^*u_{kl})=\frac{1}{2}(\delta_{ij}\delta_{kl}+\delta_{ij}\delta_{kl})=\delta_{ij}\delta_{kl}$, or equivalently $N_{\phi_1}=0$ and $M_{\phi_1}=I_{n^2}$. By replacing it into \eqref{eqnfondamentale}, we get $M_h=I_{n^2}$.

Consider now another state $\phi_2$ defined by $(1/2)(\delta_{A}+\delta_{\bar{A}})$ where $A=\Diag(i,1,\ldots ,1)$. We see that $M_{\phi_2}\neq I_{n^2}$ because $\phi_2(u_{22}^*u_{11})=0$. Replacing $M_{h}$ by $I_{n^2}$ and $M_\phi$ by   $M_{\phi_2}\neq I_{n^2}$ in~\eqref{eqnfondamentale} yields to a contradiction.

Now, let us remark that $\phi_1$ and $\phi_2$ are both tracial, and consequently the proof allows also to conclude that there exists no Haar trace for the boolean convolution.

\subsection{The monotone and the antimonotone case}\label{casemono}
In the proof of the nonexistence of a boolean Haar state, the only property of the boolean independence that we needed was
$$(h\star_B\phi)(abc)=\phi(a)h(b-\delta(b))\phi(c)+\delta(b)\phi(ac)$$
for $a,c$ in the right leg and $b$ in the left leg of $U_{n}^{\nc}\sqcup U_{n}^{\nc}$. The monotone independence verifies this same property and we can thus deduce that there exists no monotone Haar state. On the contrary, the antimonotone case verifies $(h\star_{AM}\phi)(abc)=h(b)\phi(ac).$
Nevertheless, for $x,z$ in the left leg and $y$ in the right leg of $U_{n}^{\nc}\sqcup U_{n}^{\nc}$, we have $$(\phi\star_{AM} h)(xyz)=\phi(x)h(y-\delta(y))\phi(z)+\delta(y)\phi(xz).$$ We can then do the computation of the relation $h(u_{ij}u_{kl}^*)=(\phi\star_{AM}h)\Delta(u_{ij}u_{kl}^*)$ in the exact same way as before and we find that $M_h=(N_\phi\otimes\bar{N}_\phi)(M_h-I_{n^2})+M_\phi$.
We again find a contradiction by looking on the particular states $\phi_1$ and $\phi_2$ . To sum it up, for $n\geq2$, there exists no monotone (resp. antimonotone) Haar state on $U\langle n\rangle$.

The same remark, about the traciality of the states used, allows us to conclude about the non-existence of a Haar trace.

\subsection{The free Haar trace}\label{freecase}In this section, we define the free Haar trace and prove that is is indeed an absorbing state for the free convolution on $U_{n}^{\nc}$ with other tracial states.

Let us first interpret the existence result of the free Haar trace on $U\langle n\rangle$ in a very concrete way as follows. Let us denote by $h$ the Haar trace of $U\langle n\rangle$ for the free convolution, and by $u=(u_{i,j})_{1\leq i,j\leq n}$ the collection of generators of $U_{n}^{\nc}$. Let $A=(a_{ij})_{1\leq i,j\leq n}\in \mathcal{M}_n(\mathcal{A})$ be a collection of random variables in $(\mathcal{A},\phi)$ ($\phi$ tracial) such that $(a_{ij})_{1\leq i,j\leq n}$ is unitary. Setting
$(b_{ij})_{1\leq i,j\leq n}=uA\in  \mathcal{M}_n(\mathcal{A}\sqcup U_{n}^{\nc})$ and $(c_{ij})_{1\leq i,j\leq n}=Au\in  \mathcal{M}_n(\mathcal{A}\sqcup U_{n}^{\nc})$, the collection $\{b_{ij}\}_{1\leq i,j\leq n}$ and $\{c_{ij}\}_{1\leq i,j\leq n}$ have both the same distribution as $\{u_{ij}\}_{1\leq i,j\leq n}$ in the noncommutative space $(\mathcal{A}\sqcup U_{n}^{\nc},\phi\ast h)$.

In order to define the state which will play the role of the Haar trace, we have to define a Haar unitary variable. A noncommutative variable $U$ of a noncommutative probability space $(\mathcal A,\phi)$ is called \emph{Haar unitary} if it is a unitary variable, and $\phi(U^k)=0$ for all $k\geq 0$. Here is a description of its free cumulants.

\begin{prop}[Remark 3.4.3. of~\cite{Speicher1998}]\label{unitaire}
Let $U$ be a Haar unitary element on some noncommutative probability space. Then, for all $r\geq 1$ and $\epsilon_1,\ldots,\epsilon_r\in \{1,\ast\}$, we have:
\begin{eqnarray*}
\kappa_r(U^{\epsilon_1},\ldots,U^{\epsilon_r})&=&\left\{\begin{array}{cc}
(-1)^{r/2-1}C_{r/2-1}&\text{if }r\text{ is even and the }\epsilon_i\text{ are alternating }(\epsilon_i\neq \epsilon_{i+1})\\
0&\text{else},
\end{array}\right.
\end{eqnarray*}
where $C_i=(2i)!/(i+1)!i!$ designate the Catalan numbers.
\end{prop}
Let us consider a Haar unitary random variable $U$ in $(\mathcal{A},\phi)$ and construct from there a quantum variable $j_U:U_{n}^{\nc}\rightarrow E_{11}(\mathcal{A}\sqcup \mathcal{M}_n(\mathbb{C}))E_{11}$ determined by $j_U(u_{ij})=E_{1i}UE_{j1}$ for all $1\leq i,j\leq n$ as indicated in Proposition-Definition~\ref{haarstate}. We will study the state $h=[n(\phi\ast \tr_n)]\circ j_U$ on $U_{n}^{\nc}$. We compute first the free cumulants of our variables $u_{ij}$ and $u^*_{ij}$. In fact, for all $1\leq i,j\leq n$, we denote by $(u^*)_{ij}$ the generator $u^*_{ji}$. The free cumulants of $u_{ij}$ and $(u^*)_{ij}$ turn out to be more convenient than the free cumulants of $u_{ij}$ and $u^*_{ij}$.
\begin{cor}The free cumulants of $(u_{ij})_{1\leq i,j\leq n}$  and $((u^*)_{ij})_{1\leq i,j\leq n}=(u_{ji}^*)_{1\leq i,j\leq n}$ in the noncommutative probability space $(U_{n}^{\nc},h)$ are given as follows.

Let $1\leq i_1,j_1,\ldots,i_r,j_r\leq n$ and $ \epsilon_1,\ldots, \epsilon_r$ be either $\emptyset$ or $\ast$. If the indices are cyclic (i.e. if $j_{l-1}=i_l$ for $2\leq l\leq q$ and $i_1=j_r$), $r$ is even and the $\epsilon_i$ are alternating, we have\label{cyclicu}
\begin{eqnarray*}
\kappa_r\Big((u^{\epsilon_1})_{i_1j_1},\ldots,(u^{\epsilon_r})_{i_rj_r}\Big)&=&n^{1-r}(-1)^{r/2-1}C_{r/2-1}.
\end{eqnarray*} 
If not, the left handside is equal to zero.
\end{cor}
\begin{proof}It suffices to apply Proposition~\ref{cyclic} to $U^{(1)}=U$ and $U^{(2)}=U^*$ in order to get the free cumulants of $j_U(u_{ij})=U_{ij}$ and $j_U((u^*)_{ij})=(U^*)_{ij}$.
\end{proof}

We will need another property of free cumulants. Let us first introduce new notation. For all $r\in \mathbb{N}$, $S\subset\left\{1, \ldots , r\right\}$, $\sigma \in NC(S)$, and $A_1,\ldots,A_r\in \mathcal{A}$, set
\begin{align}\phi_\sigma\left(A_1,\ldots,A_r\right)&=\hspace{-0.4cm}\prod_{\{i_1\leq \ldots\leq i_k\}\in \sigma}\hspace{-0.4cm}\phi(A_{i_1}\cdots A_{i_k}),\nonumber\\
\kappa_\sigma\left(A_1,\ldots,A_r\right)&=\hspace{-0.4cm}\prod_{\{i_1\leq \ldots\leq i_k\}\in \sigma}\hspace{-0.4cm}\kappa_k(A_{i_1},\ldots,A_{i_k}).\label{notationcum}\end{align}
Remark that, even if we write $r$ variables on the left side, the right side only involves the variables which correspond to indices which are in $S\subset \left\{1, \ldots , r\right\}$.
\begin{prop}
Let $\left\{1,\ldots,r\right\}=E\cup F$ be a disjoint union of two subsets. We suppose that $\sigma$ is a non-crossing partition on $E$. Then, for all $A_1,\ldots, A_n\in \mathcal{A}$, we have
$$\sum_{\substack{\mu\in NC(F)\text{s.t.}\\\mu\cup\sigma\in NC(r)}}\kappa_\mu(A_1,\ldots,A_{r})=\phi_{K(\sigma)}(A_1,\ldots,A_{r})$$
where $K(\sigma)$ is the biggest partition on $F$ such that $\sigma\cup K(\sigma)$ is non-crossing.\label{separer}
\end{prop}
\begin{proof}Let us compute
$$
\phi_{K(\sigma)}(A_1,\ldots,A_{r})=\sum_{\substack{\mu\in NC(F)\\\mu\preceq K(\sigma)}}\kappa_\mu(A_1,\ldots,A_{r})
=\sum_{\substack{\mu\in NC(F)\\\mu\cup\sigma\in NC(r)}}\kappa_\mu(A_1,\ldots,A_{r})
$$
because, by definition of $K(\sigma)$, the set $\{\mu\in NC(F):\mu\preceq K(\sigma)\}$ is in one-to-one correspondence with the set $\{\mu\in NC(F):\mu\cup\sigma\in NC(r)\}$.\end{proof}

We are now ready to prove that $h=[n(\phi\ast \tr_n)]\circ j_U$ is indeed a Haar trace for the free convolution.

\begin{proof}[Proof of Theorem~\ref{Haartracetheorem} in the free case]
Let $\phi$ be a tracial state on $U_{n}^{\nc}$. Let $1\leq i_1,j_1,\ldots,i_r,j_r\leq n$, let $ \epsilon_1,\ldots, \epsilon_r$ be either $\emptyset$ or $\ast$ and set$$m=(u^{\epsilon_1})_{i_1j_1}\ldots (u^{\epsilon_r})_{i_rj_r}$$where we recall that $((u)_{ij})_{1\leq i,j\leq n}=(u_{ij})_{1\leq i,j\leq n}$  and $((u^*)_{ij})_{1\leq i,j\leq n}=(u_{ji}^*)_{1\leq i,j\leq n}$ by convention. Remark that we prefer to work with the word $m$ instead of the word $u^{\epsilon_1}_{i_1j_1}\ldots u^{\epsilon_r}_{i_rj_r}$, since the computations are easier.

Let us compute $(h\star_F \phi)(m).$
We have $(h\star_F \phi)(m)=(h\ast\phi)\circ\Delta(m)$ and
$$\Delta(u_{ij})=\sum_{k=1}^n u_{ik}^{(1)}u_{kj}^{(2)},\ \text{and}\ \Delta((u^*)_{ij})=\Delta(u^*_{ji})=\sum_{k=1}^n (u^*)_{ik}^{(2)}(u^*)_{kj}^{(1)},$$
where the exponent $(1)$ and $(2)$ indicate if the element is in the first leg of $U_{n}^{\nc}\sqcup U_{n}^{\nc}$ or in the second leg. So, when computing $\Delta(m)$, we obtain something of the form $\sum_{k_1,\ldots,k_r} m_{k_1,\ldots,k_r}$ where $m_{k_1,\ldots,k_r}$ are words of length $2r$ of the form $(u^{\epsilon_1})_{i_1k_1}(u^{\epsilon_1})_{k_1j_1}\cdots (u^{\epsilon_r})_{i_rk_r}(u^{\epsilon_r})_{k_rj_r}$ with the generators coming from both legs of $U_{n}^{\nc}\sqcup U_{n}^{\nc}$. More precisely, let us decompose $\{1,\ldots,2r\}=S\cup T$ where $S$ contains the positions of the generators which are in the first leg and $T$ contains the positions of the generators which are in the second leg, according to
\begin{align*}S&=\{2i-1:1\leq i \leq n,\epsilon(i)=\emptyset \}\cup\{2i:1\leq i \leq n,\epsilon(i)=\ast \},\\
T&=\{1,\ldots,2r\}\setminus S.
\end{align*}
We can develop the computation using the freeness of the legs:
\begin{align*}
&\hspace{-1cm}(h\ast\phi)\circ\Delta(m)=(h\ast\phi)\left(\sum_{k_1,\ldots,k_r} m_{k_1,\ldots,k_r}\right)\\
=&\sum_{k_1,\ldots,k_r\ }\sum_{\substack{\ \sigma\in NC(S),\ \mu\in NC(T)\\ \text{s.t.}\ \sigma\cup\mu\in NC(2r)}}\kappa^h_\sigma\Big((u^{\epsilon_1})_{i_1k_1},(u^{\epsilon_1})_{k_1j_1},\ldots ,(u^{\epsilon_r})_{i_rk_r},(u^{\epsilon_r})_{k_rj_r}\Big)\\
&\hspace{5cm}\cdot \kappa_\mu^\phi\Big((u^{\epsilon_1})_{i_1k_1},(u^{\epsilon_1})_{k_1j_1},\ldots ,(u^{\epsilon_r})_{i_rk_r},(u^{\epsilon_r})_{k_rj_r}\Big),\end{align*}
where we recall that, according to~\eqref{notationcum}, the free cumulant $\kappa^h_\sigma(\cdots)$ only involves the variables which correspond to indices in $S$ and $\kappa_\mu^\phi(\cdots)$ only involves the variables which correspond to indices in $T$.

Using Corollary~\ref{cyclicu}, we know that, whenever the $\epsilon_i$ are alternating and the indices are cyclic within the blocks of $\sigma\in NC(S)$, the quantity
$\kappa^h_\sigma((u^{\epsilon_1})_{i_1k_1},(u^{\epsilon_1})_{k_1j_1},\ldots ,(u^{\epsilon_r})_{i_rk_r},(u^{\epsilon_r})_{k_rj_r}$
does not depend on the indices $k_1,\ldots, k_r$. We denote it by $\kappa_\sigma^h$, and compute
\begin{align*}
&\hspace{-0.5cm}(h\ast\phi)\circ\Delta(m)\\
=&\sum_{\substack{\sigma\in NC(S),\ \mu\in NC(T)\\\text{s.t. } \sigma\text{ alternates the }\epsilon_i\ \\ \text{and }\sigma\cup\mu\in NC(2r)}}\sum_{\substack{k_1,\ldots,k_r\\\ \text{s.t. the indices are cyclic}\\\text{within eack block of }\sigma}}\kappa_\sigma^h\cdot  \kappa^\phi_\mu\Big((u^{\epsilon_1})_{i_1k_1},(u^{\epsilon_1})_{k_1j_1},\ldots ,(u^{\epsilon_r})_{i_rk_r},(u^{\epsilon_r})_{k_rj_r}\Big)
\end{align*}
Thanks to Proposition~\ref{separer}, we can sum over $\mu$ and we obtain
\begin{align}
&\hspace{-0.5cm}(h\ast\phi)\circ\Delta(m)\label{importanteqn}\\
=&\sum_{\substack{\sigma\in NC(S)\\\text{s.t. } \sigma\text{ alternates the }\epsilon_i}}\kappa_\sigma^h\cdot \sum_{\substack{k_1,\ldots,k_r\\\text{s.t. the indices are cyclic}\\\text{within eack block of }\sigma}} \phi_{K(\sigma)}\Big((u^{\epsilon_1})_{i_1k_1},(u^{\epsilon_1})_{k_1j_1},\ldots ,(u^{\epsilon_r})_{i_rk_r},(u^{\epsilon_r})_{k_rj_r}\Big).\nonumber
\end{align}
So let us now examine equation \eqref{importanteqn} in greater details. Because the blocks of $\sigma$ alternate the $\epsilon_i$, the blocks of $K(\sigma)$ must also alternate the $\epsilon_i$. One can convince himself on a few examples, but also find a full proof in Proposition~7.7. of~\cite{Nica1997}.
Now one has to understand how the cyclicity of the indices $i_1k_1,k_1j_1,\ldots ,i_rk_r,k_rj_r$ in the blocks of $\sigma$ is translated in terms of the blocks of $K(\sigma)$. For every block $B=\{r(1)\leq \ldots \leq r(q)\}$, we say that $r(1)$ and $r(q)$ are opposites in $B$.

A condition $k_q=k_{q'}$ for some $1\leq q<q'\leq r$ appears twice. Once in the case where $2q$ and $2q'-1$ are opposites in the same block of $\sigma$, which is equivalent to the fact that $2q-1$ and $2q'$ are consecutive in the same block of $K(\sigma)$ (it corresponds to the case $\epsilon_q=\ast$ and $\epsilon_{q'}=\emptyset$, see the Figure~\ref{astastunun}).\begin{figure}[h]
\begin{pspicture}(0,-1.004111)(5.094111,1.03)
\definecolor{color96}{rgb}{0.047058823529411764,0.054901960784313725,0.11372549019607843}
\definecolor{color198b}{rgb}{0.050980392156862744,0.06274509803921569,0.12156862745098039}
\psdots[dotsize=0.18,linecolor=color96](0.84,0.25)
\psdots[dotsize=0.18,linecolor=color96](1.44,0.25)
\psdots[dotsize=0.18,linecolor=color96](3.64,0.25)
\psdots[dotsize=0.18,linecolor=color96](4.24,0.25)
\psdots[dotsize=0.08,linecolor=color96](2.24,0.25)
\psdots[dotsize=0.08,linecolor=color96](2.44,0.25)
\psdots[dotsize=0.08,linecolor=color96](2.64,0.25)
\psdots[dotsize=0.08,linecolor=color96](2.84,0.25)
\psdots[dotsize=0.08,linecolor=color96](4.64,0.25)
\psdots[dotsize=0.08,linecolor=color96](4.84,0.25)
\psdots[dotsize=0.08,linecolor=color96](5.04,0.25)
\psdots[dotsize=0.08,linecolor=color96](0.44,0.25)
\psdots[dotsize=0.08,linecolor=color96](0.24,0.25)
\psdots[dotsize=0.08,linecolor=color96](0.04,0.25)
\usefont{T1}{ptm}{m}{n}
\rput(0.8203125,0.555){(2)}
\usefont{T1}{ptm}{m}{n}
\rput(1.4203124,0.555){(1)}
\usefont{T1}{ptm}{m}{n}
\rput(3.6203125,0.555){(1)}
\usefont{T1}{ptm}{m}{n}
\rput(4.220312,0.555){(2)}
\psline[linewidth=0.04cm,linecolor=color96,fillcolor=color198b](1.44,0.25)(1.44,-0.35)
\psline[linewidth=0.04cm,linecolor=color96,fillcolor=color198b](1.44,-0.35)(2.04,-0.35)
\psline[linewidth=0.04cm,linecolor=color96,fillcolor=color198b](3.64,0.25)(3.64,-0.35)
\psline[linewidth=0.04cm,linecolor=color96,fillcolor=color198b](3.64,-0.35)(3.04,-0.35)
\psdots[dotsize=0.08,linecolor=color96](2.64,-0.35)
\psdots[dotsize=0.08,linecolor=color96](2.44,-0.35)
\psdots[dotsize=0.08,linecolor=color96](2.24,-0.35)
\psdots[dotsize=0.08,linecolor=color96](2.84,-0.35)
\psdots[dotsize=0.08,linecolor=color96](0.84,0.25)
\psline[linewidth=0.04cm,linecolor=color96,fillcolor=color198b,linestyle=dashed,dash=0.16cm 0.16cm](0.84,0.25)(0.84,-0.95)
\psline[linewidth=0.04cm,linecolor=color96,fillcolor=color198b,linestyle=dashed,dash=0.16cm 0.16cm](0.84,-0.95)(2.04,-0.95)
\psline[linewidth=0.04cm,linecolor=color96,fillcolor=color198b,linestyle=dashed,dash=0.16cm 0.16cm](3.04,-0.95)(4.24,-0.95)
\psline[linewidth=0.04cm,linecolor=color96,fillcolor=color198b,linestyle=dashed,dash=0.16cm 0.16cm](4.24,0.25)(4.24,-0.95)
\psdots[dotsize=0.08,linecolor=color96](2.24,-0.95)
\psdots[dotsize=0.08,linecolor=color96](2.44,-0.95)
\psdots[dotsize=0.08,linecolor=color96](2.64,-0.95)
\psdots[dotsize=0.08,linecolor=color96](2.84,-0.95)
\psdots[dotsize=0.08,linecolor=color96](0.84,-0.95)
\psline[linewidth=0.04cm,linecolor=color96,fillcolor=color198b,linestyle=dashed,dash=0.16cm 0.16cm](0.84,-0.95)(0.24,-0.95)
\psline[linewidth=0.04cm,linecolor=color96,fillcolor=color198b,linestyle=dashed,dash=0.16cm 0.16cm](4.24,-0.95)(4.84,-0.95)
\usefont{T1}{ptm}{m}{n}
\rput(0.8504201,0.845){*}
\usefont{T1}{ptm}{m}{n}
\rput(1.4504201,0.845){*}
\end{pspicture} 
\caption{\label{astastunun}The case $k_q=k_{q'}$ in $\cdots (u^*)_{i_{q}k_{q}}^{(2)}(u^*)_{k_{q}j_{q}}^{(1)}\cdots u_{i_{q'}k_{q'}}^{(1)}u_{k_{q'}j_{q'}}^{(2)}\cdots $}
(the continuous line represents $\sigma$ while the dashed line represents $K(\sigma)$)
\end{figure}
The other case is when $2q-1$ and $2q'$ are consecutive in the same block of $\sigma$, which is equivalent to the fact that $2q$ and $2q'-1$ are opposites in the same block of $K(\sigma)$ (it corresponds to the case $\epsilon_q=\emptyset$ and $\epsilon_{q'}=\ast$, see Figure~\ref{astastdeuxdeux}).
\begin{figure}[h]
\begin{pspicture}(0,-1.2691112)(6.934111,1.295)
\psdots[dotsize=0.12](0.06,0.605)
\psdots[dotsize=0.12](0.46,0.605)
\psdots[dotsize=0.12](0.86,0.605)
\psdots[dotsize=0.2](1.46,0.605)
\psdots[dotsize=0.2](2.06,0.605)
\psdots[dotsize=0.12](2.86,0.605)
\psdots[dotsize=0.12](3.26,0.605)
\psdots[dotsize=0.12](3.66,0.605)
\psdots[dotsize=0.12](4.06,0.605)
\psdots[dotsize=0.2](4.86,0.605)
\psdots[dotsize=0.2](5.46,0.605)
\psdots[dotsize=0.12](6.06,0.605)
\psdots[dotsize=0.12](6.46,0.605)
\psdots[dotsize=0.12](6.86,0.605)
\usefont{T1}{ptm}{m}{n}
\rput(1.4403125,0.91){(1)}
\usefont{T1}{ptm}{m}{n}
\rput(2.0403125,0.91){(2)}
\usefont{T1}{ptm}{m}{n}
\rput(4.8403125,0.91){(2)}
\usefont{T1}{ptm}{m}{n}
\rput(5.4403124,0.91){(1)}
\usefont{T1}{ptm}{m}{n}
\rput(4.6990623,1.11){*}
\usefont{T1}{ptm}{m}{n}
\rput(5.2990623,1.11){*}
\psline[linewidth=0.04cm](1.46,0.605)(1.46,-1.195)
\psline[linewidth=0.04cm](1.46,-1.195)(5.46,-1.195)
\psline[linewidth=0.04cm](5.46,0.605)(5.46,-1.195)
\psline[linewidth=0.04cm](1.46,-1.195)(0.86,-1.195)
\psline[linewidth=0.04cm](5.46,-1.195)(6.06,-1.195)
\psline[linewidth=0.04cm,linestyle=dashed,dash=0.16cm 0.16cm](2.06,0.605)(2.06,-0.595)
\psline[linewidth=0.04cm,linestyle=dashed,dash=0.16cm 0.16cm](2.06,-0.595)(2.86,-0.595)
\psline[linewidth=0.04cm,linestyle=dashed,dash=0.16cm 0.16cm](4.0,-0.595)(4.86,-0.595)
\psline[linewidth=0.04cm,linestyle=dashed,dash=0.16cm 0.16cm](4.86,0.605)(4.86,-0.595)
\psdots[dotsize=0.12](0.46,-1.195)
\psdots[dotsize=0.12](0.06,-1.195)
\psdots[dotsize=0.12](6.46,-1.195)
\psdots[dotsize=0.12](6.86,-1.195)
\psdots[dotsize=0.12](3.265889,-0.595)
\psdots[dotsize=0.12](3.685889,-0.595)
\end{pspicture} 
\caption{\label{astastdeuxdeux}The case $k_q=k_{q'}$ in $\cdots u_{i_{q}k_{q}}^{(1)}u_{k_{q}j_{q}}^{(2)}\cdots (u^*)_{i_{q'}k_{q'}}^{(2)}(u^*)_{k_{q'}j_{q'}}^{(1)}\cdots $
}(the continuous line represents $\sigma$ while the dashed line represents $K(\sigma)$)
\end{figure}

Now, let us consider one block $B=\{r(1)\leq \ldots \leq r(q)\}\in K(\sigma)$. If $\epsilon(r(1))=\ast$, a case illustrated in Figure~\ref{caseast}, we have
\begin{align*}&\phi_B\Big((u^{\epsilon_1})_{i_1k_1},(u^{\epsilon_1})_{k_1j_1},\ldots ,(u^{\epsilon_r})_{i_rk_r},(u^{\epsilon_r})_{k_rj_r}\Big)\\
&=\phi\Big((u^*)_{i_{r(1)}k_{r(1)}}u_{k_{r(2)}j_{r(2)}}\cdots(u^*)_{i_{r(q-1)}k_{r(q-1)}} u_{k_{r(q)}j_{r(q)}}\Big)\\
&=\phi\Big(u^*_{k_{r(1)}i_{r(1)}}u_{k_{r(2)}j_{r(2)}}\cdots u^*_{k_{r(q-1)}i_{r(q-1)}} u_{k_{r(q)}j_{r(q)}}\Big)
\end{align*}
and summing over the indices $k_{r(1)}=k_{r(2)},k_{r(3)}=k_{r(4)},\cdots, k_{r(q-1)}=k_{r(q)}$ yields to
$$\delta_{i_{r(1)}j_{r(2)}}\ldots  \delta_{i_{r(q-1)}j_{r(q)}}.$$
\begin{figure}[h]
\begin{pspicture}(0,-1.09925)(8.430312,1.09925)
\definecolor{color96}{rgb}{0.047058823529411764,0.054901960784313725,0.11372549019607843}
\definecolor{color228b}{rgb}{0.050980392156862744,0.06274509803921569,0.12156862745098039}
\psdots[dotsize=0.12,linecolor=color96](0.1853125,0.36075002)
\psdots[dotsize=0.12,linecolor=color96](0.7853125,0.36075002)
\psdots[dotsize=0.12,linecolor=color96](2.3853126,0.36075002)
\psdots[dotsize=0.12,linecolor=color96](2.9853125,0.36075002)
\psdots[dotsize=0.12,linecolor=color96](7.5853124,0.36075002)
\psdots[dotsize=0.12,linecolor=color96](8.185312,0.36075002)
\psline[linewidth=0.04cm,linecolor=color96,fillcolor=color228b,linestyle=dashed,dash=0.16cm 0.16cm](0.1853125,0.36075002)(0.1853125,-1.03925)
\psline[linewidth=0.04cm,linecolor=color96,fillcolor=color228b,linestyle=dashed,dash=0.16cm 0.16cm](0.1853125,-1.03925)(2.9853125,-1.03925)
\psline[linewidth=0.04cm,linecolor=color96,fillcolor=color228b,linestyle=dashed,dash=0.16cm 0.16cm](2.9853125,0.36075002)(2.9853125,-1.03925)
\psline[linewidth=0.04cm,linecolor=color96,fillcolor=color228b,linestyle=dashed,dash=0.16cm 0.16cm](2.9853125,-1.03925)(4.1853123,-1.03925)
\psline[linewidth=0.04cm,linecolor=color96,fillcolor=color228b,linestyle=dashed,dash=0.16cm 0.16cm](6.7853127,-1.03925)(8.185312,-1.03925)
\psline[linewidth=0.04cm,linecolor=color96,fillcolor=color228b,linestyle=dashed,dash=0.16cm 0.16cm](8.185312,0.36075002)(8.185312,-1.03925)
\psline[linewidth=0.04cm,linecolor=color96,fillcolor=color228b](0.7853125,0.36075002)(0.7853125,-0.43925)
\psline[linewidth=0.04cm,linecolor=color96,fillcolor=color228b](0.7853125,-0.43925)(2.3853126,-0.43925)
\psline[linewidth=0.04cm,linecolor=color96,fillcolor=color228b](2.3853126,0.36075002)(2.3853126,-0.43925)
\psline[linewidth=0.04cm,linecolor=color96,fillcolor=color228b](7.5853124,0.36075002)(7.5853124,-0.43925)
\psline[linewidth=0.04cm,linecolor=color96,fillcolor=color228b](7.5853124,-0.43925)(6.9853125,-0.43925)
\psdots[dotsize=0.08,linecolor=color96](5.9853125,-1.03925)
\psdots[dotsize=0.08,linecolor=color96](5.7853127,-1.03925)
\psdots[dotsize=0.08,linecolor=color96](5.5853124,-1.03925)
\psdots[dotsize=0.08,linecolor=color96](5.3853126,-1.03925)
\psdots[dotsize=0.08,linecolor=color96](5.1853123,-1.03925)
\psdots[dotsize=0.08,linecolor=color96](4.9853125,-1.03925)
\psdots[dotsize=0.08,linecolor=color96](4.9853125,-0.43925)
\psdots[dotsize=0.08,linecolor=color96](5.1853123,-0.43925)
\psdots[dotsize=0.08,linecolor=color96](5.3853126,-0.43925)
\psdots[dotsize=0.08,linecolor=color96](5.5853124,-0.43925)
\psdots[dotsize=0.08,linecolor=color96](5.7853127,-0.43925)
\psdots[dotsize=0.08,linecolor=color96](5.9853125,-0.43925)
\psdots[dotsize=0.08,linecolor=color96](6.1853123,-0.43925)
\psdots[dotsize=0.08,linecolor=color96](6.3853126,-0.43925)
\usefont{T1}{ptm}{m}{n}
\rput(0.18046875,0.66575){(2)}
\usefont{T1}{ptm}{m}{n}
\rput(0.78046876,0.66575){(1)}
\usefont{T1}{ptm}{m}{n}
\rput(2.3804688,0.66575){(1)}
\usefont{T1}{ptm}{m}{n}
\rput(2.9804688,0.66575){(2)}
\usefont{T1}{ptm}{m}{n}
\rput(7.5804687,0.66575){(1)}
\usefont{T1}{ptm}{m}{n}
\rput(8.180469,0.66575){(2)}
\rput(0.1853125,0.96075){*}
\rput(0.7853125,0.96075){*}
\psdots[dotsize=0.12,linecolor=color96](4.3853126,0.36075002)
\psdots[dotsize=0.12,linecolor=color96](3.7853124,0.36075002)
\psline[linewidth=0.04cm,linecolor=color96,fillcolor=color228b](4.3853126,0.36075002)(4.3853126,-0.43925)
\psline[linewidth=0.04cm,linecolor=color96,fillcolor=color228b](4.3853126,-0.43925)(4.7853127,-0.43925)
\psline[linewidth=0.04cm,linecolor=color96,fillcolor=color228b](7.5853124,-0.43925)(6.5853124,-0.43925)
\psline[linewidth=0.04cm,linecolor=color96,fillcolor=color228b,linestyle=dashed,dash=0.16cm 0.16cm](3.7853124,0.36075002)(3.7853124,-1.03925)
\usefont{T1}{ptm}{m}{n}
\rput(3.7804687,0.66575){(2)}
\usefont{T1}{ptm}{m}{n}
\rput(4.380469,0.66575){(1)}
\rput(3.7853124,0.96075){*}
\rput(4.3853126,0.96075){*}
\end{pspicture} 
\caption{\label{caseast}The case $k_{r(1)}=k_{r(2)},k_{r(3)}=k_{r(4)},\cdots, k_{r(q-1)}=k_{r(q)}$ in $B$}(the continuous line represents $\sigma$ while the dashed line represents $K(\sigma)$)
\end{figure}

As well, if $\epsilon(r(1))=\emptyset$, a case illustrated in Figure~\ref{caseempty}, we have by using the fact that $\phi$ is tracial,
\begin{align*}&\phi_B\Big((u^{\epsilon_1})_{i_1k_1},(u^{\epsilon_1})_{k_1j_1},\ldots ,(u^{\epsilon_r})_{i_rk_r},(u^{\epsilon_r})_{k_rj_r}\Big)\\
&=\phi\Big(u_{k_{r(1)}j_{r(1)}}(u^*)_{i_{r(2)}k_{r(2)}}\cdots u_{k_{r(q-1)}j_{r(q-1)}}(u^*)_{i_{r(q)}k_{r(q)}}\Big)\\
&=\phi\Big(u^*_{k_{r(q)}i_{r(q)}}u_{k_{r(1)}j_{r(1)}}u^*_{k_{r(2)}i_{r(2)}}\cdots u_{k_{r(q-1)}j_{r(q-1)}}\Big)
\end{align*}
and summing over the indices $k_{r(q)}=k_{r(1)},k_{r(2)}=k_{r(3)},\cdots, k_{r(q-2)}=k_{r(q-1)}$ yields to
$$\delta_{i_{r(q)}j_{r(1)}}\ldots  \delta_{i_{r(q-2)}j_{r(q-1)}}.$$
\begin{figure}[h]
\begin{pspicture}(0,-1.4992499)(9.630313,1.4992499)
\definecolor{color379}{rgb}{0.047058823529411764,0.054901960784313725,0.11372549019607843}
\definecolor{color428b}{rgb}{0.050980392156862744,0.06274509803921569,0.12156862745098039}
\psdots[dotsize=0.12,linecolor=color379](0.1853125,0.76075)
\psdots[dotsize=0.12,linecolor=color379](0.7853125,0.76075)
\psdots[dotsize=0.12,linecolor=color379](2.3853126,0.76075)
\psdots[dotsize=0.12,linecolor=color379](2.9853125,0.76075)
\psdots[dotsize=0.12,linecolor=color379](4.7853127,0.76075)
\psdots[dotsize=0.12,linecolor=color379](5.3853126,0.76075)
\psdots[dotsize=0.12,linecolor=color379](8.785313,0.76075)
\psdots[dotsize=0.12,linecolor=color379](9.385312,0.76075)
\usefont{T1}{ptm}{m}{n}
\rput(0.18046875,1.06575){(1)}
\usefont{T1}{ptm}{m}{n}
\rput(0.78046876,1.06575){(2)}
\usefont{T1}{ptm}{m}{n}
\rput(2.3804688,1.06575){(2)}
\usefont{T1}{ptm}{m}{n}
\rput(2.9804688,1.06575){(1)}
\usefont{T1}{ptm}{m}{n}
\rput(4.780469,1.06575){(1)}
\usefont{T1}{ptm}{m}{n}
\rput(5.380469,1.06575){(2)}
\usefont{T1}{ptm}{m}{n}
\rput(8.780469,1.06575){(2)}
\usefont{T1}{ptm}{m}{n}
\rput(9.380468,1.06575){(1)}
\psline[linewidth=0.04cm,linecolor=color379,fillcolor=color428b,linestyle=dashed,dash=0.16cm 0.16cm](0.7853125,0.76075)(0.7853125,-0.63925)
\psline[linewidth=0.04cm,linecolor=color379,fillcolor=color428b,linestyle=dashed,dash=0.16cm 0.16cm](2.3853126,0.76075)(2.3853126,-0.63925)
\psline[linewidth=0.04cm,linecolor=color379,fillcolor=color428b,linestyle=dashed,dash=0.16cm 0.16cm](0.7853125,-0.63925)(2.3853126,-0.63925)
\psline[linewidth=0.04cm,linecolor=color379,fillcolor=color428b,linestyle=dashed,dash=0.16cm 0.16cm](2.3853126,-0.63925)(5.3853126,-0.63925)
\psline[linewidth=0.04cm,linecolor=color379,fillcolor=color428b,linestyle=dashed,dash=0.16cm 0.16cm](5.3853126,0.76075)(5.3853126,-0.63925)
\psline[linewidth=0.04cm,linecolor=color379,fillcolor=color428b,linestyle=dashed,dash=0.16cm 0.16cm](5.3853126,-0.63925)(5.9853125,-0.63925)
\psline[linewidth=0.04cm,linecolor=color379,fillcolor=color428b,linestyle=dashed,dash=0.16cm 0.16cm](7.9853125,-0.63925)(8.785313,-0.63925)
\psline[linewidth=0.04cm,linecolor=color379,fillcolor=color428b,linestyle=dashed,dash=0.16cm 0.16cm](8.785313,0.76075)(8.785313,-0.63925)
\psline[linewidth=0.04cm,linecolor=color379,fillcolor=color428b](0.1853125,0.76075)(0.1853125,-1.43925)
\psline[linewidth=0.04cm,linecolor=color379,fillcolor=color428b](0.1853125,-1.43925)(5.9853125,-1.43925)
\psline[linewidth=0.04cm,linecolor=color379,fillcolor=color428b](7.9853125,-1.43925)(9.385312,-1.43925)
\psline[linewidth=0.04cm,linecolor=color379,fillcolor=color428b](9.385312,0.76075)(9.385312,-1.43925)
\psline[linewidth=0.04cm,linecolor=color379,fillcolor=color428b](2.9853125,0.76075)(2.9853125,-0.039249994)
\psline[linewidth=0.04cm,linecolor=color379,fillcolor=color428b](2.9853125,-0.039249994)(4.7853127,-0.039249994)
\psline[linewidth=0.04cm,linecolor=color379,fillcolor=color428b](4.7853127,0.76075)(4.7853127,-0.039249994)
\psdots[dotsize=0.08,linecolor=color379](6.3853126,-1.43925)
\psdots[dotsize=0.08,linecolor=color379](6.5853124,-1.43925)
\psdots[dotsize=0.08,linecolor=color379](6.7853127,-1.43925)
\psdots[dotsize=0.08,linecolor=color379](6.9853125,-1.43925)
\psdots[dotsize=0.08,linecolor=color379](7.1853123,-1.43925)
\psdots[dotsize=0.08,linecolor=color379](7.3853126,-1.43925)
\psdots[dotsize=0.08,linecolor=color379](7.5853124,-1.43925)
\psdots[dotsize=0.08,linecolor=color379](7.7853127,-1.43925)
\psdots[dotsize=0.08,linecolor=color379](6.1853123,-0.63925)
\psdots[dotsize=0.08,linecolor=color379](6.3853126,-0.63925)
\psdots[dotsize=0.08,linecolor=color379](6.5853124,-0.63925)
\psdots[dotsize=0.08,linecolor=color379](6.7853127,-0.63925)
\psdots[dotsize=0.08,linecolor=color379](6.9853125,-0.63925)
\psdots[dotsize=0.08,linecolor=color379](7.1853123,-0.63925)
\psdots[dotsize=0.08,linecolor=color379](7.3853126,-0.63925)
\psdots[dotsize=0.08,linecolor=color379](7.5853124,-0.63925)
\psdots[dotsize=0.08,linecolor=color379](7.7853127,-0.63925)
\rput(2.3853126,1.36075){*}
\rput(2.9753125,1.36075){*}
\rput(8.715313,1.36075){*}
\rput(9.385312,1.36075){*}
\end{pspicture}
\caption{\label{caseempty}The case $k_{r(q)}=k_{r(1)},k_{r(2)}=k_{r(3)},\cdots, k_{r(q-2)}=k_{r(q-1)}$ in $B$}(the continuous line represents $\sigma$ while the dashed line represents $K(\sigma)$)
\end{figure}

Those computations shows that the quantity $(h\ast\phi)\circ\Delta(m)$ expressed as~\eqref{importanteqn} does not depend on the choice of $\phi$, and in particular, we can replace $\phi$ by $\delta$ and obtain $(h\ast\phi)\circ\Delta(m)=(h\ast\delta)\circ\Delta(m)$. Since $m$ is arbitrary, we have $(h\ast\phi)\circ\Delta=(h\ast\delta)\circ\Delta$. Now, let us remark that $h\ast\delta$ and $h\circ (\Id \sqcup \delta)$ are two unital linear functionals which vanish on products in $U_{n}^{\nc}\sqcup U_{n}^{\nc}$ which alternates elements from $\ker(h)$ in the first leg and elements from $\ker(\delta)$ in the second leg. As a consequence, we have $h\ast\delta=h\circ (\Id \sqcup \delta)$, and we can write $(h\ast\phi)\circ\Delta=(h\ast\delta)\circ\Delta=h\circ (\Id \sqcup \delta)\circ\Delta=h$. This prove that $h$ is a Haar trace, thanks to Proposition~\ref{rlhaar}.
\end{proof}

 The free Haar state can be computed with the help of the following proposition, which is just a reformulation of Corollary~\ref{cyclicu}.

\begin{prop}When $U_{n}^{\nc}$ is endowed with its Haar trace for the free convolution, the free cumulants of $\{u_{ij}\}_{1\leq i,j\leq n}$ are given as follows.

Let $1\leq i_1,j_1,\ldots,i_r,j_r\leq n$. We have\label{cyclicuu}
\begin{eqnarray*}
\kappa_r\Big(u_{i_1j_1},u^*_{i_2j_1},u_{i_2j_2},u^*_{i_3j_2},\ldots,u_{i_r j_r},u_{i_1j_r}^*\Big)&=&n^{1-2r}(-1)^{r-1}C_{r-1}\\
\text{and }\ \ \kappa_r\Big(u_{i_1j_1}^*,u_{i_1j_2},u_{i_2j_2}^*,u_{i_2j_3},\ldots,u_{i_r j_r}^*,u_{i_rj_1}\Big)&=&n^{1-2r}(-1)^{r-1}C_{r-1}
\end{eqnarray*}
where $C_r=(2r)!/(r+1)!r!$ designate the Catalan numbers. Moreover, the free cumulants which are not given in such a way are equal to $0$.
\end{prop}

In~\cite{McClanahan1992}, Mc Clanahan defines a state on $U_{n}^{\nc}$ which is in fact equal to our free Haar trace. More precisely, let us denote by $C(\mathbb{U})$ the algebra of continuous functions on the unit complex circle $\mathbb{U}$ and by $\mathcal{M}_n(\mathbb{C})'$ the relative commutant of $\mathcal{M}_n(\mathbb{C})$ in $C(\mathbb{U})\sqcup\mathcal{M}_n(\mathbb{C})$. It is straightforward to verify that there exists a unique $*$-homomorphism $\varphi:U_{n}^{\nc}\to \mathcal{M}_n(\mathbb{C})'$ such that
$$\varphi(u_{ij})=\sum_{1\leq k \leq n}E_{ki}\Id_{\mathbb{U}}E_{jk} .$$
Endowing $C(\mathbb{U})$ with the uniform measure $h$ on the unit circle gives us a state $(\tr_n\ast h)_{|\mathcal{M}_n(\mathbb{C})'}\circ \varphi$ on $U_{n}^{\nc}$.

\begin{prop}The state $(\tr_n\ast h)_{|\mathcal{M}_n(\mathbb{C})'}\circ \varphi$ of Mc Clanahan is the Haar trace for the free convolution on $U_{n}^{\nc}$.\label{McC}
\end{prop}

\begin{proof}Let us first observe the $*$-homomorphism of noncommutative probability spaces (where $\mathcal{A}=C(\mathbb{U})$ equipped with Haar measure):
$$\begin{array}{rcl}\tilde{\varphi}:\Big(E_{11}(\mathcal{A}\sqcup \mathcal{M}_n(\mathbb{C}))E_{11} ,n(\phi\ast \tr_n)\Big)& \to& \Big(\mathcal{M}_n(\mathbb{C})',(\tr_n\ast h)_{|\mathcal{M}_n(\mathbb{C})'}\Big) \\A &\mapsto & \displaystyle\sum_{1\leq k\leq n}E_{k1}AE_{1k} 
\end{array}$$
which follows from the equality $ \phi\ast \tr_n(\sum_kE_{k1}AE_{1k})=n\ \phi\ast \tr_n(A)$ for all element $A$ of $E_{11}(\mathcal{A}\sqcup \mathcal{M}_n(\mathbb{C}))E_{11}$. Observe also that $\Id_{\mathbb{U}}$ is a Haar unitary element $U$ of $(C(\mathbb{U}),h)$. The result follows from the equality $\varphi=\tilde{\varphi}\circ j_U$ which shows that the state of Mc Clanahan $(\tr_n\ast h)_{|\mathcal{M}_n(\mathbb{C})'}\circ \varphi$ is exactly the Haar trace $[n(\phi\ast \tr_n)] \circ j_U=(\tr_n\ast h)_{|\mathcal{M}_n(\mathbb{C})'}\circ \tilde{\varphi}\circ j_U$.
\end{proof}

\subsection{The tensor Haar trace}
In this section, we prove that there exists a tensor Haar trace.

Let us define the state which will be the tensor Haar trace. It is constructed via a very different method than the free Haar trace. We consider the Hilbert space $H=\ell^2(\mathbb{Z})\otimes\bigotimes_{k\in\mathbb{Z}}\mathcal{M}_n(\mathbb{C})$, where $\ell^2(\mathbb{Z})$ is Hilbert space of square-summable families of complex numbers indexed by $\mathbb{Z}$ and $\bigotimes_{k\in\mathbb{Z}}\mathcal{M}_n(\mathbb{C})$ is the infinite tensor product of copies of the Hilbert space $\mathcal{M}_n(\mathbb{C})$, where the number of matrices different from $I_n$ is finite and the scalar product on $\mathcal{M}_n(\mathbb{C})$ is given by $\tr_n(A^*B)=\Tr(A^*B)/n$.

For all $1\leq i,j\leq n$, we define the following bounded operator on $H$ by setting, for all $\delta_k\otimes\bigotimes_{l\in\mathbb{Z}}M_l\in H$,
$$U_{ij}(\delta_k\otimes(\ldots\otimes M_{k-1}\otimes M_k\otimes M_{k+1}\otimes\ldots))=\delta_{k+1}\otimes(\ldots \otimes M_{k-1}\otimes E_{ji}M_k\otimes M_{k+1}\otimes\ldots)$$
and therefore its adjoint, given by
$$U_{ij}^*(\delta_k\otimes(\ldots\otimes M_{k-1}\otimes M_k\otimes M_{k+1}\otimes\ldots))=\delta_{k-1}\otimes(\ldots\otimes E_{ij}M_{k-1}\otimes M_{k}\otimes M_{k+1}\otimes\ldots).$$
We introduce $\Omega=\delta_0\otimes\bigotimes_{k\in \Z} I_n$ and the state on the algebra $B(H)$ of bounded operators on $H$ given by
$ A\mapsto \langle \Omega,A\Omega\rangle.$
The operators $U_{ij}$ verify that $\sum_{k=1}^n U_{ki}^*U_{kj}=\delta_{ij}=\sum_{k=1}^n U_{ik}U_{jk}^*$ and so the quantum random variable over  $j:U_{n}^{\nc}\ni u_{ij}\mapsto U_{ij}\in B(H)$ is well-defined. It induces a state $h$ on $U\langle n\rangle$, given for all $a\in U_{n}^{\nc}$ by
$$h(a)= \langle \Omega,j(a)\Omega\rangle.$$
Let us compute first the value of $h$, thanks to the following lemmas.
\begin{lm}For all $1\leq i_1,j_1,\ldots ,i_r,j_r\leq n$, we have\label{alternating}
\begin{eqnarray*}
  h(u_{i_1j_1}u_{i_2j_2}^*\ldots u_{i_{r-1}j_{r-1}}u_{i_rj_r}^*)&=&\frac{1}{n}\delta_{i_1i_2}\delta_{j_2j_3}\delta_{i_3i_4}\ldots\delta_{i_{r-1}i_r}\delta_{j_rj_1},\\
  h(u_{i_1j_1}^*u_{i_2j_2}\ldots u_{i_{r-1}j_{r-1}}^*u_{i_rj_r})&=&\frac{1}{n}\delta_{j_1j_2}\delta_{i_2i_3}\delta_{j_3j_4}\ldots\delta_{j_{r-1}j_r}\delta_{i_ri_1}.
\end{eqnarray*}
\end{lm}
\begin{proof}
We have
$$ U_{i_1j_1}U_{i_2j_2}^*\ldots U_{i_{r-1}j_{r-1}}U_{i_rj_r}^*(\Omega)=\delta_0\otimes (\cdots \otimes I_n\otimes \underbrace{E_{j_1i_1}E_{i_2j_2}\ldots E_{j_{r-1}i_{r-1}}E_{i_rj_r}}_{\text{at the level }-1}\otimes I_n\otimes \cdots),$$
$$ U_{i_1j_1}^*U_{i_2j_2}\ldots U_{i_{r-1}j_{r-1}}^*U_{i_rj_r}(\Omega)=\delta_0\otimes (\cdots \otimes I_n\otimes \underbrace{E_{i_1j_1}E_{j_2i_2}\ldots E_{i_{r-1}j_{r-1}}E_{j_ri_r}}_{\text{at the level }0}\otimes I_n\otimes \cdots),$$
which yields the first and the second result.
\end{proof}
For more general words, it is possible to reduce them and fit into the previous case. Fix $1\leq i_1,j_1,\ldots,i_r,j_r\leq n$, $ \epsilon_1,\ldots, \epsilon_r\in \{\emptyset,\ast\}$, and consider the word $u_{i_1j_1}^{\epsilon_1}\ldots u_{i_rj_r}^{\epsilon_r}$. We can decompose $\{1,\ldots,r\}$ into $\bigcup_{k=-r}^rS_k$, where\begin{equation}S_k=\Big\{l\in\{1,\ldots,r\}:k=\sharp\{m>l:\epsilon_m=\emptyset\}-\sharp\{m\geq l:\epsilon_m=\ast\}\Big\}.\end{equation}If we assume that $\emptyset$ corresponds to a North step, $\ast$ to a South step, and consider the path given by $\epsilon_r,\ldots, \epsilon_1$, the set $S_k$ contains the positions where the path goes from the level $k$ to the level $k+1$, or from the level $k+1$ to the level $k$. Consequently, the $S_k$ form a partition of $\{1,\ldots,r\}$, and the $\epsilon_m$ are alternating inside each $S_k$.

\begin{lm}\label{lmrefindeptensorielle}
Let $1\leq i_1,j_1,\ldots,i_r,j_r\leq n$ and $ \epsilon_1,\ldots, \epsilon_r$ be either $\emptyset$ or $\ast$.

If $\sharp\{m:\epsilon_m=*\}\neq\sharp\{m:\epsilon_m=1\}$, then $h(u_{i_1j_1}^{\epsilon_1}\ldots u_{i_rj_r}^{\epsilon_r})=0.$

If $\sharp\{m:\epsilon_m=*\}=\sharp\{m:\epsilon_m=1\}$, then
$$h(u_{i_1j_1}^{\epsilon_1}\ldots u_{i_rj_r}^{\epsilon_r})=\prod_{k=- r}^rh\left(\prod^{\rightarrow}_{l\in S_k}u_{i_lj_l}^{\epsilon_l}\right).$$
\end{lm}
This lemma combined with Lemma~\ref{alternating} describes entirely the state $h$.

\begin{proof}
Let us prove by decreasing induction on $l$ that, for all $1\leq l\leq r$, and $l\in S_k$, we have
\begin{eqnarray*}
U_{i_lj_l}^{\epsilon_l}\ldots U_{i_rj_r}^{\epsilon_r}(\Omega)=&\displaystyle\delta_{k+1}\otimes\bigotimes_{p\in \Z}\left(\prod^\rightarrow_{q\in S_p\cap\{l,\ldots,r\}} E_{j_qi_q}^{\epsilon_q}\right)&\text{ if }\epsilon_l=1,\\
 \text{and }\ U_{i_lj_l}^{\epsilon_l}\ldots U_{i_rj_r}^{\epsilon_r}(\Omega)=&\displaystyle\delta_{k}\otimes\bigotimes_{p\in \Z}\left(\prod^\rightarrow_{q\in S_p\cap\{l,\ldots,r\}} E_{j_qi_q}^{\epsilon_q}\right)&\text{ if }\epsilon_l=*.
\end{eqnarray*}
First of all, we have $U_{i_rj_r}(\Omega)=\delta_{1}\otimes(\ldots\otimes E_{j_ri_r}\otimes\ldots)$ with the non-identity matrix at level $0$ and $U_{i_rj_r}^*(\Omega)=\delta_{-1}\otimes(\ldots\otimes E_{i_rj_r}\otimes\ldots)$ where the non-identity matrix is at level $-1$. Thus the property is true for $l=r$.

Fix now $1\leq l< r$ and assume that the property is true for $l+1$. Suppose first that $\epsilon_l=\epsilon_{l+1}=*$, and denote by $k$ the integer such that $l+1\in S_k$, then $l\in S_{k-1}$ and:
\begin{eqnarray*}
  U_{i_lj_l}^{\epsilon_l}\ldots U_{i_rj_r}^{\epsilon_r}(\Omega)&=&U_{i_lj_l}^{\epsilon_l}\left[\displaystyle\delta_{k}\otimes\bigotimes_{p\in \Z}\left(\prod^\rightarrow_{q\in S_p\cap\{l,\ldots,r\}} E_{j_qi_q}^{\epsilon_q}\right)\right]\\
  &=&\displaystyle\delta_{k-1}\otimes\bigotimes_{p\in \Z}\left(\prod^\rightarrow_{q\in S_p\cap\{l+1,\ldots,r\}} E_{j_qi_q}^{\epsilon_q}\right).
\end{eqnarray*}
The other cases (i.e., $\epsilon_l=\epsilon_{l+1}=\emptyset$, $\epsilon_l=*,\epsilon_{l+1}=\emptyset$ and $\epsilon_{l+1}=*,\epsilon_l=\emptyset$) are treated in the exact same way. Therefore, the property is true for every $l\in\{1,\ldots,r\}$.

Finally, $h(u_{i_1j_1}^{\epsilon_1}\ldots u_{i_rj_r}^{\epsilon_r})=\langle \Omega,U_{i_1j_1}^{\epsilon_1}\ldots U_{i_rj_r}^{\epsilon_r}(\Omega)\rangle$ is exactly as expected.
\end{proof}

We are now ready to prove that $h$ is indeed the Haar trace for the tensor convolution. Thanks to Proposition \ref{sidestate}, it is a consequence of the following proposition. 
\begin{prop}The state $h$ is tracial, and for all other tracial state $\phi$, we have $h\star_T \phi=h$.
\end{prop}
\begin{proof}
Firstly, $h$ is tracial. Indeed, let us fix $1\leq i_1,j_1,\ldots,i_q,j_q\leq n$, $ \epsilon_1,\ldots, \epsilon_q\in \{\emptyset,\ast\}$ and compare $h(u_{i_1j_1}^{\epsilon_1}\ldots u_{i_rj_r}^{\epsilon_r})$ with $h(u_{i_rj_r}^{\epsilon_r}u_{i_1j_1}^{\epsilon_1}\ldots u_{i_{r-1}j_{r-1}}^{\epsilon_{r-1}})$. Thanks to Lemma~\ref{alternating}, if the $\epsilon_i$ are alternating, we are done. If not, remark that acting by a cyclic permutation just shifts the $S_k$'s. Thus, up to a cyclic permutation, the decomposition in the $S_k$'s is the same for $u_{i_1j_1}^{\epsilon_1}\ldots u_{i_rj_r}^{\epsilon_r}$ and $u_{i_rj_r}^{\epsilon_r}u_{i_1j_1}^{\epsilon_1}\ldots u_{i_{r-1}j_{r-1}}^{\epsilon_{r-1}}$. Consequently, by Lemma \ref{lmrefindeptensorielle}, the full traciality is a consequence of the traciality for words alternating the $\epsilon_i$'s.

Now, let us prove that $h\star_T \phi=h$. Equivalently, we will prove that, for all $1\leq i_1,j_1,\ldots,i_r,j_r\leq n$ and $ \epsilon_1,\ldots, \epsilon_r\in \{\emptyset,\ast\}$,
\begin{equation}h\star_T\phi(u_{i_1j_1}^{\epsilon_1}\ldots u_{i_rj_r}^{\epsilon_r})=\sum_{k_1,\ldots,k_r=1}^nh(u_{i_1k_1}^{\epsilon_1}\ldots u_{i_rk_r}^{\epsilon_r})\phi(u_{k_1j_1}^{\epsilon_1}\ldots u_{k_rj_r}^{\epsilon_r})\end{equation}
is equal to $h(u_{i_1j_1}^{\epsilon_1}\ldots u_{i_rj_r}^{\epsilon_r}).$
If $\sharp\{m:\epsilon_m=*\}\neq\sharp\{m:\epsilon_m=1\}$, this is a direct consequence of Lemma~\ref{lmrefindeptensorielle}. If not, let us prove the result by induction on the even length $r=2q$ of the word. 

Remark that $h\star_T \phi(1)=h(1)=1$. Fix $q> 0$ and  suppose that the result is true for words of length less than $2q$. Let us prove that the result is true for words  $u_{i_1j_1}^{\epsilon_1}\ldots u_{i_{r}j_{r}}^{\epsilon_{r}}$ of length $r=2q$ such that $\sharp\{m:\epsilon_m=*\}=\sharp\{m:\epsilon_m=1\}$.

Fix $1\leq i_1,j_1,\ldots,i_r,j_r\leq n$ and $ \epsilon_1,\ldots, \epsilon_r\in \{\emptyset,\ast\}$. Consider $k_0=\min \{k:S_k\neq \emptyset\}$. For all $q\in S_{k_0}\setminus \{1\}$ such that $\epsilon_q=\emptyset$, we must have $\epsilon_{q-1}=\ast$ and consequently $q-1\in S_{k_0}$ (indeed, if $\epsilon_{q-1}=\emptyset$, then $q-1\in S_{k_0+1}$ and $k_0$ is not minimal). Moreover, if $1\in S_{k_0}$ and $\epsilon_1=\emptyset$, we must have $\epsilon_{r}=\ast$ and consequently $r\in S_{k_0}$ (indeed, in this case, $\epsilon_1=\emptyset$ implies that $1\in S_{-1}$ and $k_0=-1$, and if $\epsilon_{r}=\emptyset$, then $r\in S_{0}$ and $-1$ is not minimal). To sum up, $S_{k_0}$ can be written in the form $$\{r(1)\leq r(1)+1\leq r(2)\leq r(2)+1\leq\cdots \leq r(q)\leq r(q)+1\}$$ if the first element is labelled by $\ast$, and in the form $$\{1\leq r(1)\leq r(1)+1\leq\cdots \leq r(q)\leq r(q)+1,r\}$$ if the first element is labelled by $\emptyset$.

In the case where the first element is labelled by $\ast$, let us decompose $S_{k_0}=\{r(1)\leq r(1)+1\leq \cdots \leq r(q)\leq r(q)+1\}$. Set $r(q+1)=r(1)$, and compute, thanks to Lemmas~\ref{alternating} and~\ref{lmrefindeptensorielle},
\begin{align*}&h\star_T \phi(u_{i_1j_1}^{\epsilon_1}\ldots u_{i_{r}j_{r}}^{\epsilon_{r}})=\sum_{k_1,\ldots,k_{r}=1}^nh(u_{i_1k_1}^{\epsilon_1}\ldots u_{i_{r}k_{r}}^{\epsilon_{r}} )\phi(u_{k_1j_1}^{\epsilon_1}\ldots u_{k_{r}j_{r}}^{\epsilon_{r}})\\
&=\sum_{k_1,\ldots,k_{r}=1}^nh\left(\prod^{\rightarrow}_{l\notin S_{k_0}}u_{i_lk_l}^{\epsilon_l}\right)\frac{1}{n}\prod_{l=1}^q(\delta_{k_{r(l)}k_{r(l)+1}}\delta_{i_{r(l)+1}i_{r(l+1)}})\cdot \phi(u_{k_1j_1}^{\epsilon_1}\ldots u_{k_{r}j_{r}}^{\epsilon_{r}}).
\end{align*}
Summing over the indices $k_{r(1)}=k_{r(1)+1},k_{r(2)}=k_{r(2)+1},\cdots, k_{r(q)}=k_{r(q)+1}$ in $\phi$ and use the induction hypothesis yields to
\begin{align*}
&\sum_{\substack{l\notin S_{k_0}\\1\leq k_l\leq n}} h\left(\prod^{\rightarrow}_{l\notin S_{k_0}}u_{i_lk_l}^{\epsilon_l}\right)\phi\left(\prod^{\rightarrow}_{l\notin S_{k_0}}u_{k_lj_l}^{\epsilon_l}\right)\frac{1}{n}\prod_{l=1}^q(\delta_{j_{r(l)}j_{r(l)+1}}\delta_{i_{r(l)+1}i_{r(l+1)}})\\
&=h\star_T \phi\left(\prod^{\rightarrow}_{l\notin S_{k_0}}u_{i_lj_l}^{\epsilon_l}\right)\frac{1}{n}\prod_{l=1}^q(\delta_{j_{r(l)}j_{r(l)+1}}\delta_{i_{r(l)+1}i_{r(l+1)}})\\
&=h\left(\prod^{\rightarrow}_{l\notin S_{k_0}}u_{i_lj_l}^{\epsilon_l}\right)\frac{1}{n}\prod_{l=1}^q(\delta_{j_{r(l)}j_{r(l)+1}}\delta_{i_{r(l)+1}i_{r(l+1)}})\\
&=h(u_{i_1j_1}^{\epsilon_1}\ldots u_{i_{r}j_{r}}^{\epsilon_{r}}).
\end{align*}
In the case where the first element is labelled by $\emptyset$, we decompose $S_{k_0}=\{1\leq r(1)\leq r(1)+1\leq\cdots \leq r(q)\leq r(q)+1\leq r\}$. The previous computation can be written as well, with a needed shift which has to be done in order to sum over the index $k_{1}=k_{r}$:
\begin{multline*}\sum_{k_1=1}^n \phi(u_{k_1j_1}u_{k_2j_2}^{\epsilon_2}\ldots u_{k_{r-1}j_{r-1}}^{\epsilon_{r-1}} u_{k_{r}j_{r}}^\ast)\\=\sum_{k_1=1}^n \phi(u_{k_{r}j_{r}}^\ast u_{k_1j_1}u_{k_2j_2}^{\epsilon_2}\ldots u_{k_{r-1}j_{r-1}}^{\epsilon_{r-1}})=\delta_{j_1,j_r}\phi(u_{k_2j_2}^{\epsilon_2}\ldots u_{k_{r-1}j_{r-1}}^{\epsilon_{r-1}}).\end{multline*}
Finally, we always have $h \star_T \phi(u_{i_1j_1}^{\epsilon_1}\ldots u_{i_{r}j_{r}}^{\epsilon_{r}})=h(u_{i_1j_1}^{\epsilon_1}\ldots u_{i_{r}j_{r}}^{\epsilon_{r}})$ and the proof is done.
\end{proof}
\section{Random matrix models}\label{matrixmodel}
In this section, we define a model of random matrices which converges to the free Haar trace defined in Section~\ref{Haarstate}.

Let us fix an arbitrary set $I$ of indices. Let $(M_i)_{i\in I}$ be a family of random variables in some non-commutative space $(\mathcal{A},\phi)$. For each $N\in \mathbb{N}$, let $(M_i^{(N)})_{i\in I}$ be a family of random $N\times N$ matrices. We will say that $(M_i^{(N)})_{i\in I}$ \emph{converges almost surely in $*$-distribution} to $(M_i)_{i\in I}$ as $N$ tends to $\infty$ if for all noncommutative polynomial $P\in \mathbb{C}\langle X_i,X_i^*:i\in I\rangle$ we have almost surely the following convergence:
$$\lim_{N\to \infty}\tr_N \left(P(M_i^{(N)})\right)=  \phi \left(P(M_i)\right),$$
where we recall that $\tr_N$ is the normalized trace.

The following theorem, whose first version is due to Voiculescu~\cite{Voiculescu1991}, is a well-known phenomenon which makes freeness appear from independence and invariance by unitary conjugation. (see also~\cite{Collins2003,LEVY2011,Nica2006,Voiculescu1992}).

\begin{thm}[Theorem 23.14 of~\cite{Nica2006}]\label{Voi}Let $I$ and $J$ be two arbitrary set of indices. Let $(A_i)_{i\in I}$ be a family of random variables in $(\mathcal{A},\phi)$ and $(B_j)_{i\in J}$ be a family of random variables in $(\mathcal{B},\tau)$. We suppose that
\begin{enumerate}
\item For each $N\in \mathbb{N}$, $\{A_i^{(N)}\}_{i\in I}$ is a family of random $N\times N$ matrices which converges almost surely in $*$-distribution to $\{A_i\}_{i\in I}$ as $N$ tends to $\infty$.
\item For each $N\in \mathbb{N}$, $\{B_j^{(N)}\}_{i\in J}$ is a family of constant $N\times N$ matrices which converges almost surely in $*$-distribution to $\{B_j\}_{j\in J}$ as $N$ tends to $\infty$.
\item The law of $\{A_i^{(N)}\}_{i\in I}$ is invariant by unitary conjugation, i.e. it is equal to the law of $\{UM_i^{(N)}U^*\}_{i\in I}$ for all $U\in \mathcal{M}_N(\mathbb{C})$ which is unitary.
\end{enumerate}
 Then the matrices $\{A_i^{(N)}\}_{i\in I}\cup\{B_j^{(N)}\}_{i\in  J}$ converge almost surely in $*$-distribution together to $\{A_i\}_{i\in I}\cup\{B_j\}_{i\in  J}$ seen as elements of $(\mathcal{A}\sqcup \mathcal{B},\phi\ast \tau)$ as $N$ tends to $\infty$
\end{thm}

For a matrix $M\in \mathcal{M}_{nN}(\mathbb{C})$ and $1\leq i,j\leq n$, we denote by $[M]_{ij}$ the $(i,j)$-block of $M$ when it is divided in $n^2$ matrices of size $N\times N$.

\begin{cor}Let $I$ be an arbitrary set of indices. Let $(A_k)_{k\in K}$ be a family of random variables in $(\mathcal{A},\phi)$. For each $N\in \mathbb{N}$, let $\{A_k^{(N)}\}_{k\in K}$ be a family of random $N\times N$ matrices which converges almost surely in $*$-distribution to $\{A_k\}_{k\in K}$ as $N$ tends to $\infty$ and whose law is invariant by unitary conjugation.

Then, the family of block matrices $\{[A_k^{(nN)}]_{ij}\}_{k\in K,1\leq i,j\leq n}$ converges almost surely in $*$-distribution to $\{E_{1i}A_kE_{j1}\}_{k\in K,1\leq i,j\leq n}$ seen as an element of $\left(E_{11}(\mathcal{A}\sqcup \mathcal{M}_n(\mathbb{C}))E_{11},n\ \phi\ast \tr_n\right)$ when $N$ tends to $\infty$.\label{block}
\end{cor}

Let us remark that the invariance of the law by unitary conjugation is not very restrictive. Indeed, if the law of $\{A_k^{(N)}\}_{k\in K}$ is not invariant by unitary conjugation, we can replace the family $\{A_k^{(N)}\}_{k\in K}$ by the family $\{UA_k^{(N)}U^*\}_{k\in K}$, where $U$ is a uniform unitary random matrix of $\mathcal{M}_N(\mathbb{C})$ independent from $\{A_k^{(N)}\}_{k\in K}$.

\begin{proof}First, remark that the family of constant $nN\times nN$ matrices $\{P_{ij}^{(N)}\}_{1\leq i,j\leq n}$ defined by the block matrices $[P_{ij}^{(N)}]_{lm}=\delta_{il}\delta_{jm}I_{N}$ (the block of $P_{ij}^{(N)}$ are zero except the $(i,j)$-th block which is $I_N$) converges to $\{E_{ij}\}_{1\leq i,j\leq n}\subset \mathcal{M}_n(\mathbb{C})$ as $N$ tends to $\infty$. Using Theorem~\ref{Voi}, the family $\{P_{1i}^{(N)}A_k^{(nN)}P_{j1}^{(N)}\}_{k\in K,1\leq i,j\leq n}$ converges to $\{E_{1i}A_kE_{j1}\}_{k\in K,1\leq i,j\leq n}$ seen as an element of $\left(\mathcal{A}\sqcup \mathcal{M}_n(\mathbb{C}),\phi\ast \tr_n\right)$ when $N$ tends to $\infty$.

But let us remark that $P_{1i}^{(N)}A_k^{(nN)}P_{j1}^{(N)}=\left(\begin{array}{cc}[A_k^{(nN)}]_{ij} & 0 \\0 & 0\end{array}\right)$. Consequently, the morphism of algebra from $\mathcal{M}_N(\mathbb{C})$ to $\mathcal{M}_{nN}(\mathbb{C})$ given by $M\mapsto \left(\begin{array}{cc}M & 0 \\0 & 0\end{array}\right)$ and the previous convergence implies the convergence of $\{[A_k^{(nN)}]_{ij}\}_{k\in K,1\leq i,j\leq n}$ as $N$ tends to $\infty$. However, one has to be careful that the trace $\tr_N$ is transformed via this map into the linear functional $n\tr_{nN}$, and that consequently the family $\{[A_k^{(nN)}]_{ij}\}_{k\in K,1\leq i,j\leq n}$ converges to $\{E_{1i}A_kE_{j1}\}_{k\in K,1\leq i,j\leq n}$ seen as elements of $\mathcal{A}\sqcup \mathcal{M}_n(\mathbb{C})$ endowed with the linear functional $n(\phi\ast \tr_n)$, or equivalently, seen as elements of the noncommutative space $\left(E_{11}(\mathcal{A}\sqcup \mathcal{M}_n(\mathbb{C}))E_{11},n(\phi\ast \tr_n)\right)$.
\end{proof}

A \emph{Haar unitary matrix} on the unitary group $U(N)=\{M\in \mathcal{M}_N(\mathbb{C}):U^*U=I_N\}$ is a uniformly distributed unitary matrix $U^{(N)}$, or equivalently a random unitary matrix $U^{(N)}$ which is equal in law to $VU^{(N)}$ and $U^{(N)}V$ for every unitary matrix $V$.
\begin{thm}Let us consider $(u_{ij})_{1\leq i,j\leq n}$, the generators of the non-commutative space $U\langle n \rangle$ endowed with its free Haar trace. For all $N\geq 1$, let $U^{(N)}$ be a Haar unitary matrix on the classical unitary group $U(N)$.

Then, \label{limitrandomdeux}the matrices $\left([U^{(nN)}]_{ij}\right)_{1\leq i,j\leq n}$ converge almost surely in $*$-distribution to $(u_{ij})_{1\leq i,j\leq n}$ when $N$ tends to $\infty$.
\end{thm}

\begin{proof}
Setting $\{U^{(N)}_k\}_{k\in K}=\{U^{(N)},{U^{(N)}}^{\ast}\}$ (with $K=\{1,2\}$), it is a direct consequence of Corollary~\ref{block}. Indeed, $U^{(N)}$ converge almost surely to a Haar unitary random variable $U$, and the Haar trace is given by $[n(\phi\ast \tr_n)]\circ j_U)$.
\end{proof}

\section{Free Lévy processes on the unitary dual group}\label{secfree}

In this section, we study free Lévy processes on the unitary dual  group. We recall their definition and the correspondence between Lévy processes, generators, and Schürmann triples. We describe a class of free Lévy processes which appears as limit of Lévy processes on the classical unitary group, and compute their generators thanks to a representation theorem which was still missing in the free case.

\subsection{Free Lévy processes}
\begin{defi}\label{definitionLevy}
A \emph{free unitary Lévy process} is a family $(U_t)_{t\geq0}$ of unitary element of a noncommutative probability space $(\mathcal{A},\phi)$ such that:
\begin{itemize}
\item $U_0=1_\mathcal{A}$.
\item For all $0\leq s\leq t$, the distribution of $U_s^{-1}U_t$ depends only on $t-s$.
\item For all $0\leq t_1\leq\ldots\leq t_k$, the random variables $U_{t_1},U_{t_1}^{-1}U_{t_2},\ldots,U_{t_{n-1}}^{-1}U_{t_n}$ are free. 
\item The distribution of $U_t$ converges weakly to $\delta_1$ as $t$ goes to $0$.
\end{itemize}
\end{defi}

One can generalize this definition by considering a process $(U_t)_{t\geq0}$ of matrices of elements of $\mathcal{A}$ which are unitary, instead of considering only one element. In other words, we want to consider a process $(j_t)_{t\geq0}$ of quantum random variables on $U\langle n\rangle$ over $(\mathcal{A},\phi)$ (for all time $t\geq 0$, $j_t:U_{n}^{\nc}\to \mathcal{A}$ is a $*$-homomorphism, which is equivalent with requiring that the matrix $(j_t(u_{ij}))_{i,j=1}^n$ is unitary).

\begin{defi}
A \emph{free Lévy process} on $U\langle n\rangle$ over $(\mathcal{A},\phi)$ is a family of quantum random variables $(j_t)_{t \geq 0}$ on $U\langle n\rangle$ over $\mathcal{A}$ such that:
\begin{itemize}
\item $j_0=\delta 1_{\mathcal{A}}$.
\item For all $0\leq s\leq t$, $\phi\circ ((j_s\circ \Sigma)\star j_t)=\phi\circ j_{t-s}$ (stationary of the distributions).
\item For all $0\leq t_1\leq \ldots\leq t_k$, the homomorphisms $ j_{t_1},(j_{t_1}\circ \Sigma)\star j_{t_{2}},\ldots,(j_{t_n}\circ \Sigma)\star j_{t_{n-1}}$ are freely independent in the sense that the image $*$-algebras of $U_{n}^{\nc}$ are freely independent in $(\mathcal{A},\phi)$.
\item For all $b\in U_{n}^{\nc}$, $\phi\circ j_{s}(b)$ converges towards $\delta(b)$ when $s$ tends to $0$.
\end{itemize}
\end{defi}

Some authors find more convenient to make the following assumptions on the family of increments $(j_{st})_{0\leq s \leq t}$ linked with $(j_t)_{t\geq0}$ by the relation $j_{st}=(j_s\circ \Sigma)\star j_t$ (for all $0\leq s\leq t$):
\begin{itemize}
\item For all $0\leq t$, $j_{tt}=\delta 1_{\mathcal{A}}$.
\item For all $0\leq r\leq s\leq t$, $j_{rs}\star j_{st}=j_{rt}$.
\item For all $0\leq s\leq t$, $\phi\circ j_{st}=\phi\circ j_{0,t-s}$.
\item For all $0\leq t_1\leq \ldots\leq t_k$, the homomorphisms $j_{0t_1},\ldots,j_{t_{n-1}t_n}$ are freely independent in the sense that the image algebras are freely independent.
\item For all $b\in U_{n}^{\nc}$, $\phi\circ j_{0s}(b)$ converges towards $\delta(b)$ when $s$ tends to $0$.
\end{itemize}
Of course, the two points of view are equivalent. Let us observe that a free unitary Lévy process $(U_t)_{t\geq0}$ is a free Lévy process $(u\mapsto U_t)_{t\geq 0}$ on $U\langle 1\rangle$.

\subsection{Free Lévy processes as limit of random matrices}

%
Let us present here an example of of free Lévy process constructed thanks to the homomorphism $j_U$ described in Section~\ref{buildstate}, and which is the limit of random matrices in the sense of Theorem~\ref{limitrandomun}.
\begin{prop}
Let $(U_t)_{t\geq0}$ be a free unitary Lévy process.\label{levyprocess} Let us consider the family of quantum random variables $(j_t)_{t \geq 0}$ on $U\langle n\rangle$ over $E_{11}(\mathcal{A}\sqcup \mathcal{M}_n(\mathbb{C}))E_{11}$ defined by $j_t:=j_{U_t}$, or, in other words, for all $1\leq i,j\leq n$, we have $j_t(u_{ij})=E_{1i}U_tE_{j1}$.

Then, $(j_{t})_{t \geq 0}$ is a free Lévy process on $U\langle n\rangle$ over the non-commutative probability space $\Big(E_{11}(\mathcal{A}\sqcup \mathcal{M}_n(\mathbb{C}))E_{11} , n(\phi\ast \tr_n)\Big)$.
\end{prop}


\begin{proof}The fact that $(j_{t})_{0\leq t}$ is indeed a free Lévy process on $U\langle n\rangle$ follows from Proposition~\ref{freeness}, and from the definition of a free unitary Lévy process $(U_t)_{t\geq0}$.
\end{proof}

\begin{thm}Let $(U_t)_{t\geq0}$ be a free unitary Lévy process in $(\mathcal{A},\phi)$ and let $(j_{t})_{t\geq0}$ be the Lévy process over $U\langle n \rangle$ defined by Proposition~\ref{levyprocess}. For each $N\in \mathbb{N}$, let us consider a process $(U_t^{(N)})_{t\geq0}$ on the classical unitary group $U(N)$.

Assume that the family $\{U_t^{(N)}\}_{t\geq0}$ converges almost surely in $*$-distribution to the family $\{U_t\}_{t\geq0}$ as $N$ tends to $\infty$. Then, the block matrices \label{limitrandomun}
$\left([U_t^{(nN)}]_{ij}\right)_{\substack{1\leq i,j\leq n\\ t\geq0}}$ converge almost surely in $*$-distribution to $ \Big(j_{t}(u_{ij})\Big)_{\substack{1\leq i,j\leq n\\ t\geq0}}$ as $N$ tends to $\infty$. 
\end{thm}
In the particular case where $(U_t)_{t\geq0}$ is a free unitary Brownian motion (see the last section of the paper), this theorem above is the result stated in~\cite{Ulrich2014}, proved via stochastic calculus.
\begin{proof}
Setting $\{A_k^{(N)}\}_{k\in K}=\{U_t^{(N)}\}_{t\geq0}$ (with $K=\{{t\geq0}\}$), it is a direct consequence of Corollary~\ref{block}.
\end{proof}

In~\cite{Cebron2014b}, one of the authors defined a matrix model for every unitary free Lévy process $(U_t)_{t\geq0}$. More precisely, for each $N\in \mathbb{N}$, there exists a Lévy process $(U_t^{(N)})_{t\geq0}$ on the classical unitary group $U(N)$ such that the family $\{U_t^{(N)}\}_{t\geq0}$ converges almost surely in $*$-distribution to the family $\{U_t\}_{t\geq0}$. As a consequence, every free Lévy process defined according to Proposition~\ref{levyprocess} from a one-dimensional free Lévy process is indeed the limit of a family of random matrices when the dimension tends to $\infty$.

The rest of the paper is devoted to compute the generator of such free Lévy processes, whose expression is given in Theorem~\ref{thSchurmannbiz}.

\subsection{Generator and Schürmann triple}
In this section, we define two different objects which characterize Lévy processes on $U\langle n\rangle$.
\begin{defi}The \emph{generator} of a free Lévy process $(j_t)_{t \geq 0}$ on $(\mathcal{A},\phi)$ over $U\langle n\rangle$ is the linear form $L:U_{n}^{\nc} \to \mathbb{C}$ defined, for all $u\in U_{n}^{\nc}$, by
$$L(u)=\frac{\diff}{\diff t}\phi\circ j_t (u)=\lim_{t\to 0}\frac{1}{t}\left(\phi( j_t(u))-\delta(u)\right).$$
\end{defi}
In~\cite{Ghorbal2005}, it is proved that $L$ is well-defined and determines completely the family of law $(\phi\circ j_t)_{t \geq 0}$.
The generator satisfies $L(1)=0$, is hermitian and is conditionally positive, in the sense that
\begin{itemize}
\item $L(u^*)=\overline{L(u)}$ for all $u\in U_{n}^{\nc}$ ,
\item $L(u^*u)\geq 0$ for all $u\in U_{n}^{\nc}$ such that $\delta(u)=0$.
\end{itemize}
Conversely, the recent~\cite{Schurmann2014} proves that, for all hermitian and conditionally positive $L:U_{n}^{\nc}\to \C$ such that $L(1)=0$, there exists a free Lévy process on $U\langle n\rangle$ whose generator is $L$. We will call such a linear functional \emph{a generator}, without mentioning any Lévy process. The description of the generators is made easier by the following notion of Schürmann triple.
\begin{defi}A \emph{Schürmann triple} $(\rho,\eta,L)$ of $U\langle n\rangle$ on a Hilbert space $H$ consists of\begin{itemize}
\item a generator $L$,
\item a linear map $\eta:U_{n}^{\nc}\to H$ such that, for all $a,b\in U_{n}^{\nc}$, we have\begin{equation*}L(ab)=\delta(a)L(b)+ \langle \eta(a^*),\eta(b)\rangle+L(a)\delta(b),\end{equation*}
\item   a unital $*$-representation $\rho$ of $U_{n}^{\nc}$ on $H$ such that, for all $a,b\in U_{n}^{\nc}$, we have
\begin{equation*}\eta(ab)=\rho(a)\eta(b)+\eta(a)\delta(b).\end{equation*}
\end{itemize}
\end{defi}
It simplifies the data of $L$ because the three maps $\rho$, $\eta$ and $L$ of a Schürmann triple are uniquely determined by their values on the generators $\{u_{ij},u_{ij}^*\}_{1\leq i,j\leq n}$ of $U_{n}^{\nc}$. A sort of GNS-construction (see~\cite{Schurmann1990}) allows conversely to construct a Schürmann triple $(\rho,\eta,L)$ for every generator $L$.

In the next section, we will prove the following theorem, which computes the Schürmann triple of the Lévy process over $U\langle n \rangle$ defined by Proposition~\ref{levyprocess}.

\begin{thm}
Let $(U_t)_{t\geq0}$ be a free unitary Lévy process in $(\mathcal A,\phi)$ and let $(\rho,\eta,L)$ be its Schürmann triple on a Hilbert space $H$.\label{thSchurmannbiz}
Let $j_t:U\langle n \rangle \to E_{11}(\mathcal A \sqcup \mathcal{M}_n(\mathbb{C}))E_{11}$ be the Lévy process defined by setting, for all $1\leq i,j\leq n$, $j_t(u_{ij})=E_{1i}U_tE_{j1}$. 

The Schürmann triple $(\rho_n,\eta_n,L_n)$ of $(j_t)_{t\geq 0}$ on $H\otimes \mathcal{M}_n(\C)$ is given, for all $1\leq i,j\leq n$, by
\begin{multline}\rho_n(u_{ij})=\frac{1}{n}(\rho(u)-\Id_H)\otimes E_{ij}+\delta_{ij}\Id_H\otimes I_N,\\
\ \ \eta_n(u_{ij})=\eta(u)\otimes E_{ij}, \ \ \eta_n(u_{ij}^*)=\eta(u^*)\otimes E_{ij}, \ \ L_n(u_{ij})=\delta_{ij}L(u).\label{Schurmannbiz}\end{multline}
\end{thm}

As a corollary, we have a sufficient characterization for the existence of a random matrix model in terms of the generator (we believe that this condition is not necessary).

\begin{cor}Let $(j_{t})_{t \geq 0}$ be free Lévy process on $U\langle n\rangle$. Let $H$ be a Hilbert space such that the Schürmann triple $(\rho_n,\eta_n,L_n)$ of $(j_{t})_{t \geq 0}$ is given on $H\otimes \mathcal{M}_n(\C)$ by
\begin{multline*}\rho_n(u_{ij})=\frac{1}{n}(W-\Id_H)\otimes E_{ij}+\delta_{ij}\Id_H\otimes I_N,\\
\ \ \eta_n(u_{ij})=h\otimes E_{ij}, \ \ \eta_n(u_{ij}^*)=-W^*h\otimes E_{ij}, \ \ L_n(u_{ij})=(iR-\frac{1}{2}\|h\|^2_H)\delta_{ij},\end{multline*}
where $W$ is a unitary operator of $\mathcal{B}(H)$, $h\in H$ and $R\in \R$. Then, for each $N\in \mathbb{N}$, there exists a process $(U_t^{(N)})_{t\geq0}$ on the classical unitary group $U(N)$ such that the family of $N\times N$-blocks $\left([U_t^{(nN)}]_{ij}\right)_{\substack{1\leq i,j\leq n\\ t\geq0}}$ converges almost surely in $*$-distribution to $ \Big(j_{t}(u_{ij})\Big)_{\substack{1\leq i,j\leq n\\ t\geq0}}$ as $N$ tends to $\infty$.\label{corollaryschurmann}
\end{cor}

We give the proof of Corollary~\ref{corollaryschurmann} right now, and postpone the proof of Theorem~\ref{thSchurmannbiz} to the next section.

\begin{proof}Let us show that we are indeed in the situation of Theorem~\ref{thSchurmannbiz}, and that $W$, $h$ and $R$ can be read as the Schürmann triple of some Lévy process over $U\langle 1 \rangle$. This is a consequence of the following general description of the generators on $U\langle n\rangle$.

\begin{prop}[Proposition 4.4.7 of~\cite{Schurmann1993}]Let $H$ be a Hilbert space, $(h_{ij})_{1\leq i,j\leq n}\in \mathcal{M}_n(H)$ be elements of $H$, $(W_{ij})_{1\leq i,j\leq n}\in \mathcal{M}_n(\mathcal{B}(H))$ unitary and $(R_{ij})_{1\leq i,j\leq n}\in \mathcal{M}_n(\mathbb{C})$ self-adjoint. Then there exists a unique Schürmann triple $(\rho,\eta,L)$ on $H$ such that\label{hwrL}
\begin{equation}
\rho(u_{ij})=W_{ij};\ \
\eta(u_{ij})=h_{ij};\ \ \eta(u_{ij}^*)=-\sum_{k=1}^nW_{ki}^*h_{kj};\ \ L(u_{ij})=iR_{ij}-\frac{1}{2}\sum_{k=1}^n\langle h_{ki},h_{kj}\rangle_{H}.\label{Schurmanntriple}
\end{equation}
Conversely, each generator $L$ appears in a Schürmann triple $(\rho,\eta,L)$ on a Hilbert space $H$ as~\eqref{Schurmanntriple} for some $(h_{ij})_{1\leq i,j\leq n}$, $(W_{ij})_{1\leq i,j\leq n}$ unitary, and $(R_{ij})_{1\leq i,j\leq n}$ selfadjoint given by
\begin{equation}
h_{ij}=\eta(u_{ij});\ \ W_{ij}=\rho(u_{ij});\ \ R_{ij}=-i\left(L(u_{ij})+\frac{1}{2}\sum_{k=1}^n\langle \eta(u_{ki}),\eta(u_{kj})\rangle_{H}\right).\label{hwr}
\end{equation}
\end{prop}
Using this proposition for $W$, $h$ and $R$ shows that the generator $(\rho_n,\eta_n,L_n)$ can be written in the form~\eqref{Schurmannbiz} for some Schürmann triple $(\rho,\eta,L)$ on $H$. But let us consider a free unitary Lévy process $(U_t)_{t\geq0}$ with Schürmann triple $(\rho,\eta,L)$, and the Lévy process $(j_{U_t})_{t\geq 0}$ of Theorem~\ref{thSchurmannbiz} defined by setting, for all $1\leq i,j\leq n$, $j_{U_t}(u_{ij})=E_{1i}U_tE_{j1}$. Using the result~\cite[Theorem 3]{Cebron2014b}, there exists a random matrix model $(U_t^{(N)})_{t\geq0}$ on the unitary group for the Lévy process $(U_t)_{t\geq0}$, and Theorem~\ref{limitrandomun} allows us to conclude that $\left([U_t^{(nN)}]_{ij}\right)_{\substack{1\leq i,j\leq n\\ t\geq0}}$ is a random matrix model for $(j_{U_t})_{t\geq 0}$. Theorem~\ref{thSchurmannbiz} shows that $(j_{U_t})_{t\geq 0}$ has the same Schürmann triple that our Lévy process $(j_t)_{t\geq 0}$. Thus their distributions are equal, and $\left([U_t^{(nN)}]_{ij}\right)_{\substack{1\leq i,j\leq n\\ t\geq0}}$ is also a random matrix model for $(j_t)_{t\geq 0}$.
\end{proof}

\subsection{Proof of Theorem~\ref{thSchurmannbiz}}In the three next steps, we will
\begin{enumerate}
\item establish a concrete realization of any free Lévy process $(j_t)_{t\geq 0}$ on $U\langle n\rangle$ on a full Fock space, starting from any Schürmann triple;
\item show that, considering a one dimensional free Lévy process $(U_t)_{t\geq 0}$, this concrete realization behaves nicely when applying the boosting $j_{U_t}(u_{ij})=E_{1i}U_tE_{j1}$ to define a free Lévy process $(j_{U_t})_{t\geq 0}$ on $U\langle n\rangle$;
\item conclude the proof by reading the Schürmann triple directly from the stochastic equation of $(j_{U_t})_{t\geq 0}$.
\end{enumerate}
\subsection*{Step 1}
\label{representation}
In this step, we give a direct construction of a free Lévy process starting from a Schürmann triple of $U\langle n\rangle$. To achieve this purpose, we will use the free quantum stochastic calculus. We do not recall the definition of the free stochastic equations on the full Fock space, but we define now the objects involved, and we refer the reader to~\cite{Kummerer1992} and~\cite{Skeide2000} for further details.

Let us consider a Hilbert space $H$. We denote by $K$ the Hilbert space $L^{2}(\mathbb{R}, H)\simeq L^{2}(\mathbb{R})\otimes H$, and consider the full Fock space
$$\Gamma(K)=\mathbb{C}\Omega \oplus \bigoplus_{n\geq 1} K^{\otimes n}.$$
We turn $\mathcal{B}(\Gamma(K))$, the $*$-algebra of bounded operator on $\Gamma(K)$, into a noncommutative probability space by endowing it with the state $\tau(\cdot)=\langle  \Omega, (\cdot)\Omega\rangle$.
Let $h\in H$ and $t\geq 0$. The creation operator $c_t(h)\in \mathcal{B}(\Gamma(K))$ is defined by setting, for all $n\geq 0$,
$$c_t(h)(k_1\otimes\cdots \otimes k_n)=(h1_{[0,t[})\otimes k_1\otimes \cdots \otimes k_n, $$
and the annihilation operator $c_t^*(h)\in \mathcal{B}(\Gamma(K))$ is its adjoint operator. Let $W$ a bounded operator on $H$ and $t\geq 0$. The conservation operator $\Lambda_t(W)\in \mathcal{B}(\Gamma(K))$ is defined by setting, for all $n\geq 1$,
$$\Lambda_t(W) ((h1_{[r,s[})\otimes\cdots \otimes k_n)=((W(h)1_{[0,t[\cap[r,s[} )\otimes\cdots \otimes k_n)$$
and $\Lambda_t(W) (\Omega)=0$ otherwise.

The following general result is the free counterpart of the general results of Schürmann (see Section~4.4. of~\cite{Schurmann1993} for the tensor case). The free case turns out to be the only case which has not yet been written down.\label{freesto}
\begin{thm}Let $H$ be a Hilbert space, and let $(\rho,\eta,L)$ be a Schürmann triple of $U\langle n\rangle$ on the Hilbert space $H$. Then the coupled free stochastic equations\label{gegen}
\begin{equation}
\diff j_t(u_{ij})=\sum_{1\leq k \leq n}j_t(u_{ik})\Big(\diff c_t^*(\eta(u_{kj}))+\diff c_t(\eta(u_{kj}^*))+\diff \Lambda_t((\rho-\delta)(u_{kj}))+L(u_{kj})\diff t\Big)\label{stoeq}\end{equation}
for $1\leq i,j\leq n$, with initial conditions $j_0(u_{ij})=\delta_{ij} Id$, has a unique solution $\big(j_{t}(u_{ij})_{i,j=1}^n\big)_{t\geq 0}$ which extends to a free Lévy process $(j_t)_{t\geq 0}$ on $U\langle n\rangle$ with value in $(\mathcal{B}(\Gamma(L^2(\mathbb{R},H))),\tau),$ and with generator $L$.
\end{thm}


\begin{proof}The existence and uniqueness of the solution of~\eqref{stoeq} is a consequence of a very general theorem in~\cite{Skeide2000}, from which we can also deduce the extension of the solution to a free Lévy process. On the contrary, proving that $L$ is indeed the generator of this solution is not a direct consequence of~\cite{Skeide2000}, and requires some computations very similar to those of~\cite{Schurmann1990}.

The existence theorem which we will use is~\cite[Theorem~10.1]{Skeide2000}. In order to use Theorem~10.1. of~\cite{Skeide2000}, we must write the $n^2$ stochastic equations~\eqref{stoeq} as one stochastic equation involving only one variable. This is routine using the explanations of Chapter~13 of~\cite{Skeide2000}. For the convenience of the reader, we sketch the ideas: we consider the full Fock $\mathcal{M}_N(\mathbb{C})$-module $\mathcal{M}_N(\mathcal{B}(\Gamma(K)))\simeq \mathcal{B}(\C^n\otimes \Gamma(K))$. The stochastic equations~\eqref{stoeq} can be summed up into the following stochastic equation in $\mathcal{M}_N(\mathcal{B}(\Gamma(K)))$ (where $c_t$, $c_t^*$ and $\Lambda_t$ are defined accordingly)
\begin{multline}
\diff \Big(j_{t}(u_{ij})_{i,j=1}^n\Big)\\=\Big(j_{t}(u_{ij})_{i,j=1}^n\Big)\cdot \Big(\diff c_t^*(\eta(u_{ij})_{i,j=1}^n)+\diff c_t(\eta(u_{ij}^*)_{i,j=1}^n)+\diff \Lambda_t(\rho(u_{ij})_{i,j=1}^n-\Id)+L(u_{ij})_{i,j=1}^n\diff t\Big)\label{stoeqdeux}\end{multline}
with initial condition $(j_{t}(u_{ij}))_{i,j=1}^n=\Id$. Let us define $h=(h_{ij})_{1\leq i,j\leq n}$, $W=(W_{ij})_{1\leq i,j\leq n}$ unitary, and $R=(R_{ij})_{1\leq i,j\leq n}$ selfadjoint by the relation~\eqref{hwr}. The stochastic equation~\eqref{stoeqdeux} can be rewritten
\begin{equation}
\diff \Big(j_{t}(u_{ij})_{i,j=1}^n\Big)=\Big(j_{t}(u_{ij})_{i,j=1}^n\Big)\cdot \Big(\diff c_t^*(h)-\diff c_t(W^{-1}h)+\diff \Lambda_t(W-\Id)+\big(iR-\frac{1}{2}\sum_{k=1}^n\langle h_{ki},h_{kj}\rangle_{H}\big)\diff t\Big).\label{stoeqtrois}\end{equation}According to Theorem~10.1. of~\cite{Skeide2000} (see the end of~\cite[Chapter~10]{Skeide2000} to make the link with this particular case), there exists a unique solution to~\eqref{stoeqtrois} whenever $W=(W_{ij})_{1\leq i,j\leq n}$ is unitary and $R=(R_{ij})_{1\leq i,j\leq n}$ is selfadjoint, which is indeed true thanks to Proposition~\ref{hwrL}. Finally, there exists a unique solution $(j_{t}(u_{ij}))_{i,j=1}^n$ to the coupled stochastic equations~\eqref{stoeq}, and another consequence of~\cite[Theorem~10.1]{Skeide2000} is that $(j_{t}(u_{ij}))_{i,j=1}^n$ is unitary. This is sufficient to  extend $(j_{t}(u_{ij}))_{i,j=1}^n$ as a process $(j_t)_{t\geq 0}$ of quantum random variables. The stationary of the distributions is a consequence of the stationary of the underlying driven process and the freeness of the increments is a consequence of the particular underlying filtration for which $(j_{t}(u_{ij}))_{i,j=1}^n$ is adapted (see Chapter~11 of~\cite{Skeide2000} for the statements of those two facts).

It remains to prove that $L$ is indeed the generator of $(j_t)_{t\geq 0}$. Let us denote by $\mathcal{L}$ the generator of $(j_t)_{t\geq 0}$, defined for all $a\in U_{n}^{\nc}$ by \begin{equation}\mathcal{L}(a)=\frac{\diff}{\diff t}\phi\circ j_t (a).\label{defcall}\end{equation}In order to prove that $\mathcal{L}=L$, it suffices to prove first that, for all $b,c\in U_{n}^{\nc}$, we have
\begin{equation}\mathcal{L}(b^*c)=\mathcal{L}(b^*)\delta(c)+\delta(b^*)\mathcal{L}(c)+\langle \eta(b),\eta(c)\rangle,\label{ITODEUX}\end{equation}
which implies that $(\rho,\eta,\mathcal{L})$ is a Schürmann triple, and to prove secondly that $(\mathcal{L}(u_{i,j}))_{i,j=1}^n=(L(u_{i,j}))_{i,j=1}^n$, which implies that the Schürmann triples $(\rho,\eta,\mathcal{L})$ and $(\rho,\eta,L)$ are equal.

The quantum stochastic calculus allows us to write the quantum stochastic differential equation of $j_t(b^*c)$, thanks to the following result.
\begin{thm}[Corollary 9.2. of~\cite{Skeide2000}]Let $h,h'\in H$ and $W,W'\in \mathcal{B}(H)$. Let $I_t$ be one of the following four processes $t\mapsto t$, $c_t(h)$, $c^*_t(h)$ or $\Lambda_t(W)$, and $I'_t$ one of the following four processes $t\mapsto t$, $c_t(h')$, $c^*_t(h')$ or $\Lambda_t(W')$. Let $F,G,F'$ and $G'$ be adapted and bounded. For all $M_t$ and $M'_t$ such that $\diff M_t=F_t\diff I_t G_t$ and $\diff M'_t=F_t'\diff I_t' G_t'$, we have
$$\diff (M_t M'_t) =F_t\diff I_t (G_t  M'_t) +(M_t F_t')\diff I_t' G_t' +\tau(G_tF'_t) F_t\diff I_t'' G'_t,$$
where the integrator $\diff I''_t$ has to be chosen according to Itô's table (see Table~\ref{LLBM}).

\begin{table}[h]
\centering
\ra{1.3}
\begin{tabular}{@{}l|cccc@{}}\toprule
$\diff I_t\backslash  \diff I'_t$ & $\diff t$ & $\diff c_t(h')$ & $\diff c^*_t(h')$ & $\diff \Lambda_t(W')$\\ \midrule

$\diff t$ & $0$ & $0$ & $0$ & $0$\\
$\diff c_t(h)$ & $0$ & $0$ & $0$ & $0$\\
$\diff c^*_t(h)$& $0$ &$\langle h,h' \rangle \diff t$ & $0$ & $\diff c^*_t(W'^*h)$\\
$\diff\Lambda_t(W)$& $0$ & $\diff c_t(Wh')$ & $0$ & $\diff \Lambda_t(WW')$\\
 \bottomrule
\end{tabular}
\bigskip\caption{\label{LLBM}Itô's table}
\end{table}
\end{thm}

By induction, it follows that $j_t(b)$ and $j_t(c)$, which can be written as a polynomial in the operators $\{j_t(u_{i,j}),j_t(u_{i,j})^*\}_{1\leq i,j\leq n}$, satisfies a quantum stochastic differential equation. Moreover, by the previous theorem,
\begin{equation}\diff j_t(b^*c)=\diff j_t(b^*) j_t(c) + j_t(b^*)\diff j_t(c)+\diff j_t(b^*)\cdot\diff j_t(c)\label{ITOUN}\end{equation}
where the third term is computed thanks to the quantum Ito table. But in the definition~\eqref{defcall} of $\mathcal{L}$,
we are only dealing with expectations in the vacuum state $\tau$, and the $\diff c_t$-part, the $\diff c_t^*$-part and the $\diff \Lambda_t$-part are martingales under te vacuum state $\tau$. Thus we need only to compute the integrand of the $\diff t$-part of $\diff j_t(b^*c)$. This coefficient is a complex valued function in $t\in \R_+$ and its value at $t=0$ gives us $\mathcal{L}(b^*c)$. Using the initial condition, one checks that the first two terms on the right hand side of ~\eqref{ITOUN} give rise, under the vacuum state, to the first two terms on the right hand side of~\eqref{ITODEUX}. We are left with the computation of the coefficient of the $\diff t$-part of $\diff j_t(b^*)\cdot\diff j_t(c)$ at $t=0$. Because of the Ito table, this $\diff t$-part is coming from the $\diff c_t$-parts of $\diff j_t(b)$ and $\diff j_t(c)$ by the formula \begin{equation}\diff c_t^*(h)\cdot\diff c_t(h')=\langle h,h'\rangle \diff t.\label{ctito}\end{equation}
Thus we are left to compute the $\diff c_t$-parts of $\diff j_t(b)$ and $\diff j_t(c)$. Of course, we can assume that both $b$ and $c$ are monomials in $u_{ij}$ and $u_{ij}^*$. Assuming $b=u^{\epsilon_1}_{i_1,j_1}\cdots u^{\epsilon_r}_{i_r,j_r}$, we can compute from the differential equation of $j_t$ and the quantum Ito table the exact expression for the $\diff c_t$-part of $\diff j_t(b)$. For simplicity, we give here the expression of the $\diff c_t$-part of $\diff j_t(b)$ where we have already put the integrand at time $t=0$, as this will not affect the final result (notice that it allows us to replace $j_0(u_{ij})$ by $\delta(u_{ij})$, and $\sum_{k=1}^nj_0(u_{ik})\diff \Lambda_t((\rho-\delta)(u_{kj}))$ by $\diff \Lambda_t((\rho-\delta)(u_{ij}))$):
\begin{align*}
&\hspace{-0.5cm}\sum_{\substack{1\leq l\leq r\\1\leq m(1)\leq \ldots\leq m(l) \leq r }}\delta(u^{\epsilon_1}_{i_1,j_1}\cdots \hat{u}^{\epsilon_{m(1)}}_{i_{m(1)}j_{m(1)}}\cdots \hat{u}^{\epsilon_{m(l-1)}}_{i_{m(l-1)}j_{m(l-1)}}\cdots \hat{u}^{\epsilon_{m(l)}}_{i_{m(l)}j_{m(l)}}\cdots u^{\epsilon_r}_{i_r,j_r})\\
&\hspace{2cm}\cdot \diff \Lambda_t((\rho-\delta)(u^{\epsilon_{m(1)}}_{i_{m(1)}j_{m(1)}}))\cdots \diff \Lambda_t((\rho-\delta)(u^{\epsilon_{m(l-1)}}_{i_{m(l-1)}j_{m(l-1)}}))\diff c_t(\eta(u^{\epsilon_{m(l)}}_{i_{m(l)}j_{m(l)}})\\
=&\sum_{1\leq l\leq r}\diff \Lambda_t(\rho(u_{i_1j_1}^{\epsilon_{1}}))\cdots \diff \Lambda_t(\rho(u_{i_{l-1}j_{l-1}}^{\epsilon_{l-1}}))\diff c_t(\eta(u_{i_lj_l}^{\epsilon_l})\delta(u^{\epsilon_{l+1}}_{i_{l+1},j_{l+1}}\cdots u^{\epsilon_r}_{i_r,j_r})\\
=&\diff c_t(\eta(u^{\epsilon_1}_{i_1,j_1}\cdots u^{\epsilon_r}_{i_r,j_r}))=\diff c_t(\eta(b)),
\end{align*}
where the hats mean that we omit the terms in the product. Finally, using~\eqref{ctito}, the integrand of the $\diff t$-part of $(\diff j_t(b^*))\cdot(\diff j_t(c))$ at time $t=0$ is equal to $\langle \eta(b),\eta(c)\rangle$, which completes the equality~\eqref{ITODEUX}.

Now, for $1\leq i,j\leq n$, $\mathcal{L}(u_{ij})$ is given by the integrand of the $\diff t$-part of $\diff j_t(u_{ij})$ at time $t=0$. Indeed, the three others parts are martingales. This integrand is given by~\eqref{stoeq}:\begin{equation*}\mathcal{L}(u_{ij})=\sum_{k=1}^n\tau(j_0(u_{ik})L(u_{kj}))=\sum_{k=1}^n\delta(u_{ik})L(u_{kj})=L(u_{ij});\end{equation*}
and it concludes the proof.
\end{proof}
Using Proposition~\ref{hwrL}, it is possible to rewrite Theorem~\ref{gegen} without mentioning any Schürmann triple.
\begin{cor}Let $H$ be a Hilbert space, $(h_{ij})_{1\leq i,j\leq n}\in \mathcal{M}_n(H)$ be elements of $H$, $(W_{ij})_{1\leq i,j\leq n}\in \mathcal{M}_n(\mathcal{B}(H))$ unitary and $(R_{ij})_{1\leq i,j\leq n}\in \mathcal{M}_n(\mathbb{C})$ self-adjoint. Then the coupled free stochastic equations\label{gegencor}
\begin{multline*}
\diff j_t(u_{ij})=\sum_{1\leq k \leq n}j_t(u_{ik})\\\cdot \Big(\diff c_t^*(h_{kj})-\diff c_t(\sum_{l=1}^nW_{lk}^*h_{lj})+\diff \Lambda_t(W_{kj}-\delta_{kj}Id_H)+(iR_{kj}-\frac{1}{2}\sum_{l=1}^n\langle h_{lk},h_{lj}\rangle_{H})\diff t\Big)\end{multline*}
for $1\leq i,j\leq n$, with initial conditions $j_0(u_{ij})=\delta_{ij} Id$ has a unique solution $\big(j_{t}(u_{ij})_{i,j=1}^n\big)_{t\geq 0}$ which extends to a free Lévy process $(j_t)_{t\geq 0}$ on $U\langle n\rangle$ over $(\mathcal{B}(\Gamma(L^2(\mathbb{R},H))),\tau)$.
\end{cor}

\subsection*{Step 2}

%

Let $H$ be a Hilbert space, $(\rho,\eta,L)$ be a Schürmann triple of $U\langle 1 \rangle$ on $H$, and $K=L^2(\R_+,H)$. From Theorem~\ref{gegen}, we know that \begin{equation}
\diff U_t=U_t\Big(\diff c_t^*(\eta(u))+\diff c_t(\eta(u)^*)+\diff \Lambda_t((\rho-\delta)(u))+L(u)\diff t\Big)\label{defut}\end{equation}
with initial conditions $U_0=1$, has a unique solution $\big(U_t)_{t\geq 0}$ in $(\mathcal{B}(\Gamma(K)),\tau)$ which is a free Lévy process with Schürmann triple $(\rho,\eta,L)$. We consider
\begin{equation}j_t:U\langle n \rangle \to E_{11}(\mathcal{B}(\Gamma(K))\sqcup \mathcal{M}_n(\mathbb{C}))E_{11}\label{defjt}\end{equation}the free Lévy process defined by setting $j_t(u_{ij})=E_{1i}U_tE_{j1}$ as in Proposition~\ref{levyprocess}. The following theorem gives a stochastic equation on $\mathcal{B}(\Gamma(L^2(\mathbb{R},H)\otimes \mathcal{M}_n(\mathbb{C})))$ whose solution has the same distribution under the vacuum state than $j_t$.

Let us first remark that $L^2(\mathbb{R},H)\otimes \mathcal{M}_n(\mathbb{C})\simeq L^{2}(\mathbb{R}, H\otimes\mathcal{M}_n(\mathbb{C}))$. Thus, for all $h\otimes M\in H\otimes\mathcal{M}_n(\mathbb{C})$, the process $c^*_t(h\otimes M),c_t(h\otimes M) \in \mathcal{B}(\Gamma(K\otimes \mathcal{M}_n(\mathbb{C})))$ are defined as previously. Furthermore, for all $W\in \mathcal{B}(H)$ and $M\in \mathcal{M}_n(\mathbb{C})$, the conservation operator $\Lambda_t(W\otimes M)$ is defined as previously, with $M$ acting on $\mathcal{M}_n(\mathbb{C})$ by the left multiplication.

%

\begin{prop}Let $H$ be a Hilbert space and $(\rho,\eta,L)$ be a Schürmann triple of $U\langle 1 \rangle$ on $H$. Let $(U_t)_{t\geq 0}$ defined by~\eqref{defut} and $(j_t)_{t\geq 0}$ defined by~\eqref{defjt}. There exists a homomorphism of probability spaces\label{bigeq} $\rho:\Big(E_{11}(\mathcal{B}(\Gamma(K))\sqcup \mathcal{M}_n(\mathbb{C}))E_{11},n(\phi\ast \tr_n)\Big)\to \Big(\mathcal{B}(\Gamma(K\otimes \mathcal{M}_n(\mathbb{C}))),\langle\Omega,(\cdot)\Omega\rangle\Big)$ such that the free Lévy process $(J_t)_{t\geq 0}=(\rho \circ j_t)_{t\geq 0}$ is solution of the following differential equation, starting at $J_0(u_{ij})=\delta_{ij}\Id$:
\begin{multline*}
\diff J_t(u_{ij})=\sum_{1\leq k \leq n}J_t(u_{ik})\\
\cdot \left(\diff c_t^*(\eta(u)\otimes E_{kj})+\diff c_t(\eta(u^*)\otimes E_{kj})+\frac{1}{n}\diff \Lambda_t(((\rho-\delta)(u))\otimes E_{kj})+\delta_{kj}L(u)\diff t\right).\end{multline*}
\end{prop}

\begin{proof}Let us first describe the free product representation of $\mathcal{B}(\Gamma(K))\sqcup\mathcal{M}_n(\mathbb{C})$ given in~\cite{Voiculescu1992}. We consider $\mathcal{M}_n(\mathbb{C})$ acting on itself by the left multiplication. We denote by $\Gamma(K)^{\circ}$ the Hilbert space $\bigoplus_{n\geq 1} K^{\otimes n}$ and by $\mathcal{M}_n(\mathbb{C})^\circ$ the Hilbert space $\mathcal{M}_n(\mathbb{C})\ominus \mathbb{C}I_n$, in such a way that
$$\Gamma(K)=\Gamma(K)^\circ\oplus \mathbb{C}\Omega\ \text{ and }\ \mathcal{M}_n(\mathbb{C})=\mathcal{M}_n(\mathbb{C})^\circ\oplus \mathbb{C}I_n.$$
We denote by $k\mapsto k^\circ$ and $M\mapsto M^\circ$ the respective orthogonal projection of $\Gamma(K)$ onto $\Gamma(K)^\circ$ and of $\mathcal{M}_n(\mathbb{C})$ onto $\mathcal{M}_n(\mathbb{C})^\circ$. We consider the Hilbert space $\Gamma(K)\ast \mathcal{M}_n(\mathbb{C})$ given by
$$\Gamma(K)\ast \mathcal{M}_n(\mathbb{C})=\mathbb{C}\Omega\oplus \bigoplus_{m\geq 1}\left(\bigoplus_{\substack{H_1,\dots,H_n=\Gamma(K)^\circ\text{ or }\mathcal{M}_n(\mathbb{C})^\circ\\ H_i\neq H_{i+1}}} H_1\otimes \cdots \otimes H_m\right).$$
The algebra $\mathcal{B}(\Gamma(K))$ acts on $\Gamma(K)\ast \mathcal{M}_n(\mathbb{C})$ as follows: for $A\in \mathcal{B}(\Gamma(K)), k \in \Gamma(K)^\circ$ and $M\in \mathcal{M}_n(\mathbb{C})^\circ$, we have
\begin{align*}
\left(\pi(A)\right)\left(\Omega\right)&=(A\Omega)^\circ+\langle \Omega,A\Omega\rangle \Omega,\\
\left(\pi(A)\right)\left(k\otimes (\cdots) \right)& =(Ak)^\circ\otimes (\cdots)+\langle \Omega,Ak\rangle\cdot (\cdots),\\
\left(\pi(A)\right)\left(M\otimes (\cdots) \right)& =(A\Omega)^\circ\otimes M \otimes (\cdots)+\langle \Omega,A\Omega\rangle \cdot M\otimes (\cdots).
\end{align*}
Similarly, the algebra $\mathcal{M}_n(\mathbb{C})$ acts on $\Gamma(K)\ast \mathcal{M}_n(\mathbb{C})$ as follows: for $A\in \mathcal{M}_n(\mathbb{C}), k \in \Gamma(K)^\circ$ and $M\in \mathcal{M}_n(\mathbb{C})^\circ$, we have
\begin{align*}
\left(\lambda(A)\right)\left(\Omega\right)&=A^\circ+\langle I_N,A \rangle \Omega,\\
\left(\lambda(A)\right)\left(k\otimes (\cdots) \right)& =A^\circ\otimes k\otimes(\cdots)+\langle I_N,A\rangle \cdot k\otimes(\cdots),\\
\left(\lambda(A)\right)\left(M\otimes (\cdots) \right)& =(AM)^\circ \otimes (\cdots)+\langle I_N,AM\rangle \cdot (\cdots).
\end{align*}
According to~\cite[Section 1.5]{Voiculescu1992}, the $*$-homomorphism $\pi\sqcup \lambda:\Big(\mathcal{B}(\Gamma(K))\sqcup \mathcal{M}_n(\mathbb{C}),\phi\ast \tr_n\Big)\to \Big(\mathcal{B}(\Gamma(K)\ast \mathcal{M}_n(\mathbb{C})),\langle  \Omega,(\cdot) \Omega\rangle\Big)$ is a $*$-homomorphism of noncommutative probability spaces.

\begin{lm}There exists a Hilbert space isomorphism\label{isomorphism}
$$\Gamma(K)\ast \mathcal{M}_n(\mathbb{C})\to \Gamma(K\otimes \mathcal{M}_n(\mathbb{C}))\otimes \mathcal{M}_n(\mathbb{C})$$
which induces a $*$-algebra isomorphism
$$f:\mathcal{B}\Big(\Gamma(K)\ast \mathcal{M}_n(\mathbb{C})\Big)\to \mathcal{B}\Big( \Gamma(K\otimes \mathcal{M}_n(\mathbb{C}))\otimes \mathcal{M}_n(\mathbb{C})\Big).$$
\end{lm}
\begin{proof}We will use the three well-known isomorphisms
$$K\simeq K\otimes \mathbb{C}, \ (K\otimes \mathbb{C}) \oplus (K\otimes \mathcal{M}_n(\mathbb{C})^\circ)\simeq K\otimes \mathcal{M}_n(\mathbb{C})\ \text{ and }\ \mathcal{M}_n(\mathbb{C})\simeq \mathbb{C}^n\otimes \mathbb{C}^n.$$
It suffices to write
\begin{align*}
\Gamma(K)\ast \mathcal{M}_n(\mathbb{C})\hspace{-1cm}&\hspace{+1cm}\simeq\mathbb{C}\Omega\oplus \bigoplus_{m\geq 1}\left(\bigoplus_{\substack{H_1,\dots,H_m=\Gamma(K)^\circ\text{ or }\mathcal{M}_n(\mathbb{C})^\circ\\ H_i\neq H_{i+1}}} H_1\otimes \cdots \otimes H_m\right)\\
&\simeq\mathbb{C}\Omega\oplus \bigoplus_{m\geq 1}\left(\bigoplus_{\substack{k_1+1+k_2+\ldots+1+k_N=m\\k_2,\ldots, k_{N-1}\geq 1}} K^{\otimes k_1}\otimes \mathcal{M}_n(\mathbb{C})^\circ\otimes K^{\otimes k_2}\cdots \otimes \mathcal{M}_n(\mathbb{C})^\circ\otimes K^{\otimes k_N}\right)\\
&\simeq \bigoplus_{m'\geq 1}\left(\bigoplus_{H_1,\dots,H_{m'}=\mathbb{C} \text{ or }\mathcal{M}_n(\mathbb{C})^\circ} H_1\otimes K\otimes H_2\otimes \cdots \otimes K\otimes H_{m'} \right)\\
&\simeq \bigoplus_{m'\geq 1}\underbrace{ \mathcal{M}_n(\mathbb{C}) \otimes K\otimes \mathcal{M}_n(\mathbb{C})\otimes \cdots \otimes K\otimes \mathcal{M}_n(\mathbb{C})}_{\text{where }\mathcal{M}_n(\mathbb{C})\text{ appears }m'\text{ times}}\\
&\simeq \bigoplus_{m'\geq 1}\underbrace{ (\mathbb{C}^n\otimes \mathbb{C}^n) \otimes K\otimes (\mathbb{C}^n\otimes \mathbb{C}^n)\otimes \cdots \otimes K\otimes (\mathbb{C}^n\otimes \mathbb{C}^n)}_{\text{where }\mathbb{C}^n\otimes \mathbb{C}^n\text{ appears }m'\text{ times}}\\
&\simeq \mathbb{C}^n\otimes\left(\bigoplus_{m\geq 0}\underbrace{  (\mathbb{C}^n \otimes K\otimes \mathbb{C}^n)\otimes \cdots \otimes (\mathbb{C}^n\otimes K\otimes \mathbb{C}^n)}_{\text{where }\mathbb{C}^n \otimes K\otimes \mathbb{C}^n\text{ appears }m\text{ times}}\right)\otimes \mathbb{C}^n\\
&\simeq\mathbb{C}^n\otimes \Gamma(\mathbb{C}^n\otimes K\otimes \mathbb{C}^n)\otimes \mathbb{C}^n\\
&\simeq \Gamma( K\otimes \mathcal{M}_n(\mathbb{C}))\otimes \mathcal{M}_n(\mathbb{C})
\end{align*}
and to define the Hilbert space isomorphism $\Gamma(K)\ast \mathcal{M}_n(\mathbb{C})\to \Gamma(K\otimes \mathcal{M}_n(\mathbb{C}))\otimes \mathcal{M}_n(\mathbb{C})$ accordingly.
\end{proof}
Unfortunately, we do not see any way of writing $f$ directly, and for computing it, we will always follow the different steps of the proof of Lemma~\ref{isomorphism}.

We are interested in the $*$-subalgebra $E_{11}(\mathcal{B}(\Gamma(K))\sqcup \mathcal{M}_n(\mathbb{C}))E_{11}$, and it is important to remark here that its image by $f\circ (\pi\sqcup \lambda)$ is an algebra of operator which leaves the space $\Gamma( K\otimes \mathcal{M}_n(\mathbb{C}))\otimes \mathbb{C} E_{11}$ invariant (it suffices to follow each steps of the proof of Lemma~\ref{isomorphism}). Consequently, when restricted to $E_{11}(\mathcal{B}(\Gamma(K))\sqcup \mathcal{M}_n(\mathbb{C}))E_{11}$, the $*$-homomorphism $f\circ (\pi\sqcup \lambda)$ can be seen as a $*$-homomorphism
$$\rho:E_{11}(\mathcal{B}(\Gamma(K))\sqcup \mathcal{M}_n(\mathbb{C}))E_{11}\to \mathcal{B}\left(\Gamma( K\otimes \mathcal{M}_n(\mathbb{C}))\right)$$
using the trivial isomorphism $\Gamma( K\otimes \mathcal{M}_n(\mathbb{C}))\otimes \mathbb{C} E_{11}\simeq \Gamma( K\otimes \mathcal{M}_n(\mathbb{C}))$.
It is now a routine, following the steps of Lemma~\ref{isomorphism}, to verify that
\begin{itemize}
\item $n(\phi\ast \tr_n)(A)=\langle \Omega,\rho(A)\Omega\rangle$ for all $A\in E_{11}(\mathcal{B}(\Gamma(K))\sqcup \mathcal{M}_n(\mathbb{C}))E_{11}$;
\item $\rho(E_{1i}c_t^*(h )E_{j1})=c_t^*(h\otimes E_{ij})$ for all $h\in H$,
\item $\rho(E_{1i}c_t(h )E_{j1})=c_t(h\otimes E_{ij})$ for all $h\in H$,
\item and $\rho(E_{1i}\Lambda_t(W)E_{j1})=\frac{1}{n}\Lambda_t(W\otimes E_{ij})$ for all $W\in \mathcal{B}(H)$.
\end{itemize}
To conclude, let us write
\begin{align*}\diff j_t(u_{ij})=\hspace{-1cm}&\hspace{1cm}E_{1i}\diff U_tE_{j1}\\
&=E_{1i} U_t\cdot \left(\diff c_t^*(\eta(u))+\diff c_t(\eta(u)^*)+\diff \Lambda_t((\rho-\delta)(u))+L(u)\diff t\right)E_{j1}\\
&=\sum_{1\leq k \leq n}E_{1i} U_t E_{k1}\cdot E_{1k}\left(\diff c_t^*(\eta(u))+\diff c_t(\eta(u)^*)+\diff \Lambda_t((\rho-\delta)(u))+L(u)\diff t\right)E_{j1}\\
&=\sum_{1\leq k \leq n} j_t(u_{ik})\cdot E_{1k}\left(\diff c_t^*(\eta(u))+\diff c_t(\eta(u)^*)+\diff \Lambda_t((\rho-\delta)(u))+L(u)\diff t\right)E_{j1}
\end{align*}
and then apply the homomorphism $\rho$.
\end{proof}

\subsection*{Step 3}We conclude the proof of Theorem~\ref{thSchurmannbiz}. Recall that we start from a free unitary Lévy process $(U_t)_{t\geq0}$ with Schürmann triple $(\rho,\eta,L)$. Because Theorem~\ref{thSchurmannbiz} uniquely depends on the distribution of our random variables, we can without loss of generality represent $(U_t)_{t\geq0}$ as the solution of the stochastic equation~\eqref{defut}. Let $j_t:U\langle n \rangle \to E_{11}(\mathcal{B}(\Gamma(K))\sqcup \mathcal{M}_n(\mathbb{C}))E_{11}$ be the Lévy process defined by setting, for all $1\leq i,j\leq n$, $j_t(u_{ij})=E_{1i}U_tE_{j1}$ as in Proposition~\ref{levyprocess}.

We want to prove that $(\rho_n,\eta_n,L_n)$ defined by setting, for all $1\leq i,j\leq n$,
\begin{multline}\rho_n(u_{ij})=\frac{1}{n}(\rho(u)-\Id_H)\otimes E_{ij}+\delta_{ij}\Id_H\otimes I_N,\\
\ \ \eta_n(u_{ij})=\eta(u)\otimes E_{ij}, \ \ \eta_n(u_{ij}^*)=\eta(u^*)\otimes E_{ij}, \ \ L_n(u_{ij})=\delta_{ij}L(u),\label{Schurmannbizproof}\end{multline}
is the Schürmann triple of $(j_t)_{t\geq 0}$.

First of all, $(\rho_n,\eta_n, L_n)$ given by~\eqref{Schurmannbizproof} is a well-defined Schürmann triple. Indeed, defining $(h_{ij})_{1\leq i,j\leq n}$, $(W_{ij})_{1\leq i,j\leq n}$ unitary, and $(R_{ij})_{1\leq i,j\leq n}$ selfadjoint by
\begin{multline*}W_{ij}=\frac{1}{n}(\rho(u)-\Id_H)\otimes E_{ij}+\delta_{ij}\Id_H\otimes I_N,\\
\ \ h_{ij}=\eta(u)\otimes E_{ij}, \ \ R_{ij}=-i\left(\delta_{ij}L(u)+\frac{1}{2}\sum_{k=1}^n\langle h_{ki},h_{kj}\rangle_{H\otimes \mathcal{M}_n(\C)}\right),\end{multline*}
we can apply Proposition~\ref{hwrL} and conclude that $(\rho_n,\eta_n, L_n)$ is a Schürmann triple whenever $\eta(u^*)\otimes E_{ij}=-\sum_{k=1}^nW_{ki}^*h_{kj}$ (because in that case the relations~\eqref{Schurmanntriple} and \eqref{Schurmannbizproof} are the same). Let us verify this fact:
\begin{multline*}-\sum_{k=1}^nW_{ki}^*h_{kj}=-\frac{1}{n}\sum_{k=1}^n\left(\rho(u)^*\eta(u)\otimes E_{ik}E_{kj}-\eta(u)\otimes E_{ik}E_{kj}\right)-\eta(u)\otimes E_{ij}\\=-\rho(u)^*\eta(u)\otimes E_{ij}=\eta(u^*)\otimes E_{ij}.\end{multline*}
Proposition~\ref{bigeq} gives us the stochastic equation which drives the process $(j_t)_{t\geq 0}$ (or at least a process which has the same distribution):
\begin{multline*}
\diff j_t(u_{ij})=\sum_{1\leq k \leq n}j_t(u_{ik})\\
\cdot \left(\diff c_t^*(\eta(u)\otimes E_{kj})+\diff c_t(\eta(u^*)\otimes E_{kj})+\frac{1}{n}\diff \Lambda_t(((\rho-\delta)(u))\otimes E_{kj})+\delta_{kj}(L(u))\diff t\right),\end{multline*}
or equivalently,
\begin{equation}
\diff j_t(u_{ij})=\sum_{1\leq k \leq n}j_t(u_{ik})\Big(\diff c_t^*(\eta_n(u_{kj}))+\diff c_t(\eta_n(u_{kj}^*))+\diff \Lambda_t((\rho_n-\delta)(u_{kj}))+L_n(u_{kj})\diff t\Big).\end{equation}
Theorem~\ref{gegen} allows us to conclude that $(\rho_n,\eta_n, L_n)$ is the Schürmann triple of $(j_t)_{t\geq 0}$, which concludes the proof of Theorem~\ref{thSchurmannbiz}.

\subsection{An example: the free unitary Brownian motion}The free unitary Brownian motion introduced in~\cite{Biane1997} is the unique solution $(U_t)_{t\geq 0}$ in $\mathcal{B}(\Gamma(L^2(\R,\C))) $, starting at $U_0=\Id$, of the free stochastic equation
$$
\diff U_t=iU_t(\diff c_t^*(1)+\diff c_t(1))-\frac{1}{2}U_t\diff t,$$
or equivalently, of the equation $
\diff U_t=U_t\cdot (\diff c_t^*(-i)+\diff c_t(i)-\frac{1}{2}\diff t).$ It corresponds to a Lévy process over $U\langle 1 \rangle$ given by $(u\mapsto U_t)_{t\geq 0}$, and from Theorem~\ref{gegen}, we know that the Schürmann triple $(\rho,\eta, L)$ on $\C$ of this process is given by $\rho(u)=\Id_\C$, $\eta(u)=-\eta(u^*)=-i$ and $L(u)=-\frac{1}{2}$. Concretely, using the definition of a Schürmann triple, it means that, for all polynomial $P\in \C[X]$,
$$\frac{\diff}{\diff t}\tau(P(U_t))=L(P(u))=-\frac{1}{2}P'(1)-P''(1).$$
The free Lévy process $j_t:U\langle n \rangle \to E_{11}(\mathcal{B}(\Gamma(K))\sqcup \mathcal{M}_n(\mathbb{C}))E_{11}$ defined by $j_t(u_{ij})=E_{1i}U_tE_{j1}$ is then (thanks to Proposition~\ref{bigeq}), equal in distribution to the solution $(J_t)_{t\geq 0}$ of
\begin{align}
\diff J_t(u_{ij})&=\sum_{1\leq k \leq n}J_t(u_{ik})
\cdot \left(\diff c_t^*( -iE_{kj})+\diff c_t(iE_{kj})-\frac{1}{2}\delta_{kj}\diff t\right)\nonumber\\
&=i\sum_{1\leq k \leq n}J_t(u_{ik})
\cdot \left(\diff c_t^*(E_{kj})+\diff c_t(E_{kj})\right)-\frac{1}{2}J_t(u_{ij})\diff t,\label{gaussianprocess}
\end{align}
the Lévy process on $U\langle n\rangle$ under study in~\cite{Ulrich2014}. Theorem~\ref{limitrandomun} gives the same conclusion as in~\cite{Ulrich2014}: because the Brownian motion $(U_t^{(N)})_{t\geq 0}$ on the unitary group $U(N)$ defined and studied in~\cite{Biane1997} converges in $*$-distribution to the free unitary Brownian motion $(U_t)_{t\geq 0}$ as $N$ tends to $\infty$, the $N\times N$-block matrices $\left([U_t^{(nN)}]_{ij}\right)_{\substack{1\leq i,j\leq n\\ t\geq0}}$ converge almost surely in $*$-distribution to $ \Big(J_{t}(u_{ij})\Big)_{\substack{1\leq i,j\leq n\\ t\geq0}}$ as $N$ tends to $\infty$.

Theorem~\ref{thSchurmannbiz} shows that the representation $\rho_n$ in the Schürmann triple $(\rho_n,\eta_n,L_n)$ of $(J_t)_{0\leq s\leq t}$ on $\mathcal{M}_n(\C)$ is equal to $\delta \cdot \Id_{\mathcal{M}_n(\C)}$, which means that $(J_t)_{t\geq 0}$ is a gaussian process on $U\langle n\rangle$ (in the sense of~\cite{Franz2006}). Moreover, this process is non-degenerate in the following sense:
\begin{prop}\label{haarlimite}
Let $(J_t)_{t\geq 0}$ be the Lévy process on $U\langle n\rangle$ defined by~\eqref{gaussianprocess}. Then, when $t$ goes to infinity, the distribution of $(J_t)_{t\geq 0}$ converges towards the free Haar trace.
\end{prop}

\begin{proof}
Let $(U_t)_{t\geq 0}$ be a free multiplicative Brownian motion in a non-commutative probability space $(\mathcal A,\Phi)$. Then, $(J_t)_{t\geq 0}$ is equal in distribution to $j_t:U\langle n \rangle \to E_{11}(\mathcal A\sqcup \mathcal{M}_n(\mathbb{C}))E_{11}$ defined by setting, for all $1\leq i,j\leq n$, $j_t(u_{ij})=E_{1i}U_tE_{j1}$.

It is well-known that $(U_t)_{t\geq 0}$ converge in $*$-distribution to a Haar unitary variable $U$ as $t$ tends to $\infty$. Indeed, there is an explicit description of the moments of $U_t$ in~\cite{Biane1997}, namely
$$\tau(U_t^k)=e^{-\frac{kt}{2}}\sum_{i=0}^{k-1}(-1)^i\frac{t^i}{i!}k^{i-1}\binom{k}{i+1},\ \ k\geq 1,$$and they converge to zero, which are the moments of a Haar unitary variable $U$. As a consequence, $(j_t(u_{ij}))_{1\leq i,j\leq n}$ converge in $*$-distribution to $(E_{1i}UE_{j1})_{1\leq i,j\leq n}$ as $t$ tends to $\infty$, where $(E_{ij})_{1\leq i,j\leq n}$ are free from $U$. But $u_{ij}\mapsto E_{1i}UE_{j1}$ is a quantum random variable whose distribution is the free Haar trace (see Section~\ref{freecase}). Consequently, $(j_t)_{t\geq 0}$ converge in distribution to the free Haar trace, and so do $(J_t)_{t\geq 0}$.
\end{proof}

\bibliography{dualgroup}
\bibliographystyle{plain}
\end{document}